\documentclass{amsart}
\usepackage{pst-all}
\usepackage{amsmath}
\usepackage{amsfonts}
\usepackage{amssymb}
\usepackage{multicol}
\usepackage{verbatim}
\usepackage{amsthm}
\usepackage{amscd}
\usepackage{pb-diagram}
\usepackage[mathscr]{eucal}
\usepackage{stmaryrd}

\usepackage{mathscinet}
\usepackage[T1]{fontenc}

\usepackage{graphics}

\usepackage{doi}
\usepackage{hyperref}
\hypersetup{
    colorlinks,
    citecolor=blue,
    filecolor=black,
    linkcolor=red,
    urlcolor=blue
}

\newtheorem{theorem}{Theorem}[section]
\newtheorem{lemma}[theorem]{Lemma}
\newtheorem{corollary}[theorem]{Corollary}
\newtheorem{proposition}[theorem]{Proposition}
\newtheorem{definition}[theorem]{Definition}
\newtheorem{conjecture}[theorem]{Conjecture}
\newtheorem{remark}[theorem]{Remark}

\newtheorem{example}[theorem]{Example}

\numberwithin{equation}{section}
    \newcommand{\N}{\mathbb{N}}
    \newcommand{\Z}{\mathbb{Z}}
    
    \newcommand{\Ci}{\mathbb{C}}
    
    \newcommand{\catO}{\mathcal{O}}

    \newcommand{\Hom}{{\rm{Hom}}}

    \newcommand{\g}{\mathfrak{g}}
    
    \newcommand{\hatg}{\hat{\mathfrak{g}}}

    \newcommand{\sltwo}{\mathfrak{sl}_2}

    \newcommand{\qbinom}[2]{\genfrac{[}{]}{0pt}{}{#1}{#2}}

    \newcommand{\Uqghat}{\mathcal{U}_q\hat{\mathfrak{g}}}
    \newcommand{\Uqb}{\mathcal{U}_q\mathfrak{b}}
    \newcommand{\Uqh}{\mathcal{U}_q\mathfrak{h}}
    \newcommand{\Uqhaffine}{\mathcal{U}_q\tilde{\mathfrak{h}}}
    \newcommand{\Uqlg}{\mathcal{U}_q(\mathcal{L}\mathfrak{g})}
    \newcommand{\bpsi}{\mathbf{\Psi}}
    \newcommand{\ringY}{\mathbb{Z}[Y_{i,aq^n}^{\pm 1}]_{i \in I,n \in \mathbb{Z}}}
    \newcommand{\ringA}{\mathbb{Z}[A_{i,aq^n}^{\pm 1}]_{i \in I,n \in \mathbb{Z}}}

    \newcommand{\ringYa}{\mathbb{Z}[Y_{i,a}^{\pm 1}]_{i \in I,a \in \mathbb{C}^*}}
    \newcommand{\ringAa}{\mathbb{Z}[A_{i,a}^{\pm 1}]_{i \in I,a \in \mathbb{C}^*}}

\begin{document}

\title{Weyl group twists and representations of quantum affine Borel algebras}

\author{Keyu Wang}

\address{Universit\'e Paris Cit\'e and Sorbonne Universit\'e, CNRS, IMJ-PRG, F-75006, Paris, France}

\email{keyu.wang@imj-prg.fr}

%\date{\today}
\begin{abstract}
    We define categories $\catO^w$ of representations of Borel subalgebras $\Uqb$ of quantum affine algebras $\Uqghat$, which come from the category $\catO$ twisted by Weyl group elements $w$. We construct inductive systems of finite-dimensional $\Uqb$-modules twisted by $w$, which provide representations in the category $\catO^w$. We also establish a classification of simple modules in these categories $\catO^w$.
    
    We explore convergent phenomenon of $q$-characters of representations of quantum affine algebras, which conjecturally give the $q$-characters of representations in $\catO^w$.

    Furthermore, we propose a conjecture concerning the relationship between the category $\catO$ and the twisted category $\catO^w$, and we propose a possible connection with shifted quantum affine algebras.
\end{abstract}

\maketitle
\tableofcontents

%----------------------------------------
%SECTION
%----------------------------------------
\section{Introduction}
The action of the Weyl group $W$ on a finite-dimensional simple Lie algebra $\g$ reveals crucial symmetries of the characters of their representations. For finite-dimensional representations $V$ of the Lie algebra $\g$, their characters $\chi(V)$ remain invariant under the Weyl group action. However, this invariance does not extend to the infinite-dimensional representations in the category $\catO_{\mathfrak{g}}$, such as Verma modules.

Recent interest in the Weyl group action related to these infinite-dimensional representations has been motivated by the works of E.~Frenkel and D.~Hernandez \cite{frenkel2022weyl}. More precisely, the motivation and difficulties lie at the level of quantum affine algebras $\Uqghat$.

It is important to note that on quantum affine algebras $\Uqghat$, the actions of the Weyl group are replaced by those of the corresponding braid group $\mathcal{B}$. These are known as Lusztig's automorphisms $T_{i,\pm1}'$, $T_{i,\pm1}''$ of quantum groups \cite{lusztig1993introduction}. One natural question is how this braid group symmetry on algebras manifests on the characters of their representations.

For representations $V$ of quantum affine algebras $\Uqghat$, we have the usual characters $\chi(V)$ which correspond to the action of the Cartan subalgebra $\Uqh$. These usual characters have well-behaved symmetry under the Weyl group action. In the representation theory of quantum affine algebras, we are particularly interested in the concept of $q$-characters, which correspond to the action of the affine Cartan subalgebra $\Uqhaffine$ and contain more information than the usual characters. However, Lusztig's automorphisms do not preserve the subalgebra $\Uqhaffine$. Despite this, the $q$-characters $\chi_q(V)$ still exhibit certain symmetries under the braid group action.

V.~Chari defined a braid group action on the affine Cartan subalgebra $\Uqhaffine$, and this action provides some information on the $q$-characters \cite{chari2002braid}. For example, for a finite-dimensional irreducible representation $V$ of $\Uqghat$, the terms in its $q$-character $\chi_q(V)$ with extremal weights are given by Chari's braid group action. However, this action does not reflect the symmetry of the entire $q$-character $\chi_q(V)$. 

Recently, E.~Frenkel and D.~Hernandez discovered a Weyl group symmetry on $q$-characters \cite{frenkel2022weyl}. They also generalize this Weyl group action to the category $\catO$ of infinite-dimensional representations of Borel subalgebras $\Uqb$ \cite[Section~3.6]{frenkel2023extended}. They conjectured that this Weyl group action enables the calculation of $q$-characters of some representations in $\catO$ that were previously unknown \cite[Conjecture~4.8]{frenkel2023extended}.

More precisely, E.~Frenkel and D.~Hernandez have introduced a family of topological complete rings $\bar{\mathscr{Y}}^{'w}$ labeled by Weyl group elements $w \in W$, and they have constructed a Weyl group action on the direct sum of these rings. In this article, we explore some new $\Uqb$-modules in other categories twisted from $\catO$, whose $q$-characters belong to the ring $\bar{\mathscr{Y}}^{'w}$. As remarked in \cite[Remark~4.9]{frenkel2023extended}, it is expected that the $q$-characters of these new $\Uqb$-modules are related to the $q$-characters of $\Uqb$-modules in $\catO$ by the Weyl group action of Frenkel-Hernandez. 

We define new categories $\catO^w$ of representations of $\Uqb$ by conditions similar to the category $\catO$, but with weights dominated by a different set of negative roots $w(\Phi_-)$. In these categories, we classify simple modules and construct certain representations as limits of finite-dimensional representations.

The difficulty is that there is no known Weyl group action or braid group action on Borel subalgebras. Consequently, there is no natural way to derive representations in $\catO^w$ from those in $\catO$, and it is unknown to the author whether there is an equivalence of categories between $\catO$ and $\catO^w$. Furthermore, the actions of the algebra on representations in the categories $\catO^w$ differ from those in $\catO$. For example, the triangular decomposition of Borel subalgebras, which was crucial for studying the category $\catO$, has not been found an analogue for $\catO^w$.

Although representations in $\catO^w$ and in $\catO$ are not directly connected, their characters exhibit a close relationship. We conjecture that the usual characters $\chi(V)$ of simple representations $V$ in $\catO^w$ can be derived from those in $\catO$ through the standard Weyl group action.

Our study extends to their $q$-characters as well. From the study of $T$-systems, it is known that the normalized $q$-characters of certain simple finite-dimensional representations of $\Uqghat$, known as Kirillov-Reshetikhin modules, form a convergent sequence in a topological ring \cite{kuniba2002canonical,nakajima2003t,hernandez2006kirillov}. We hope to generalize this convergent phenomenon twisted by any Weyl group element, by introducing what we term projected limits, which will be defined in this article. The convergence of normalized $q$-characters of Kirillov-Reshetikhin modules has an algebraic explanation by Borel subalgebras \cite{hernandez2012asymptotic}. We also conjecture that the $q$-characters for some representations in $\catO^w$ can be calculated using the convergent projected limits, which gives an algebraic explanation of the convergent phenomenon twisted by Weyl group elements.

Finally, we conjecture that $q$-characters of representations in $\catO^w$ can be factorized into a product of a constant part with a non-constant part. We suggest two further research directions on this conjecture. Firstly, we examine the relationship between the $q$-character of a representation in $\catO^w$ and the representation in $\catO$ classified by the same highest weight. This conjectural relation offers an alternative method for calculating the $q$-characters of certain representations in $\catO$, including those mentioned in \cite[Conjecture~4.8]{frenkel2023extended}. Secondly, we explore the potential link between our conjecture and shifted quantum affine algebras.

The structure of this article is organized as follows:

In Section~\ref{sec:preliminaries}, we review basic definitions and some results on quantum affine algebras.

In Section~\ref{sec:limits of q characters}, we define various topological completions of the ring of $q$-characters, which are related to those considered in \cite{frenkel2023extended}. We define a notion of projected normalized characters, and we propose a conjecture regarding the convergent phenomenon of the projected normalized characters (Conjecture~\ref{conj:limit of projected w norm}).

In Section~\ref{sec:inductive systems}, we construct inductive systems and projective systems of finite-dimensional representations twisted by Lusztig's automorphisms. We demonstrate that the inductive limits and projective limits are equipped with a well-defined action of Borel subalgebras (Theorem~\ref{thm:inductive system}).

In Section~\ref{sec:Borel subalgebras}, we define and explore the categories $\catO^w$ of representations of Borel subalgebras, where the limits constructed in the previous section lie in. Despite the absence of a triangular decomposition of Borel subalgebras for studying $\catO^w$, we still manage to classify simple modules in the category $\catO^w$ using a non-standard method (Theorem~\ref{thm:classification catOw}). In Subsection~\ref{sec:conjectures on characters}, we study the usual characters as well as the $q$-characters of representations in $\catO^w$. We propose conjectures concerning the characters and their connection to projected limits discussed in Section~\ref{sec:limits of q characters} (Conjecture~\ref{conj:inductive limit and projected limit}). 

Finally, in Section~\ref{sec:further direction}, we suggest potential further research directions, including the relationship between $\catO$ and $\catO^w$ and the connection to shifted quantum affine algebras.

\subsection*{Acknowledgment}
The author wishes to express sincere gratitude to David Hernandez, who engaged in frequent discussions with the author, offering numerous insights that greatly enriched this work. Special thanks are also extended to Edward Frenkel for his valuable comments and for suggesting intriguing directions for further research.

%----------------------------------------------
%SECTION
%----------------------------------------------
\section{Preliminaries}\label{sec:preliminaries}
%----------------------subsection-------------------------%
\subsection{Quantum affine algebras}
Let $\mathfrak{g}$ be a finite-dimensional simple Lie algebra, and $\hat{\mathfrak{g}}$ be the corresponding non-twisted affine Lie algebra. The set of vertices of the finite Dynkin diagram is denoted by $I$, and the set of vertices of the affine Dynkin diagram is denoted by $\hat{I} = I \sqcup \{0\}$. Let $\widehat{C} = (C_{ij})_{i,j \in \hat{I}}$ be the Cartan matrix of $\hatg$. The Cartan matrix of $\g$ is the submatrix of $\widehat{C}$ with indices in $I$, and we denote this Cartan matrix by $C$. Let $d_i \in \N$ for $i \in \hat{I}$ so that the diagonal matrix $D = \mathrm{diag}(d_i)_{i \in \hat{I}}$ makes $D\widehat{C}$ symmetric.

Let $q \in \mathbb{C}^*$ not be a root of unity. For $m \in \N$, we use notation
\[q_i = q^{d_i}, \; [m]_{q_i} = \frac{q_i^m-q_i^{-m}}{q_i-q_i^{-1}}, \; [m]_{q_i}! = [m]_{q_i} [m-1]_{q_i} \cdots [1]_{q_i}, \quad \forall i \in \hat{I},\]
and we use the convention $[0]_q = 0$, $[0]_q! = 1$.

The quantum affine algebra $\Uqghat$ is the associative $\Ci$-algebra generated by $k_{i}^{\pm 1} , e_i,f_i \; (i \in \hat{I})$ with the quantum Weyl relations
	\begin{equation}\label{eq:quantumweylrelations}
		\begin{split}
			&k_i k_i^{-1} = k_i^{-1}k_i = 1, \; k_ik_j = k_j k_i,\\
			&k_i e_j = q_i^{\widehat{C}_{i,j}}e_j k_i, \; k_if_j = q_i^{-\widehat{C}_{i,j}}f_j k_i,\\
			&[e_i,f_j] = \delta_{i,j} \frac{k_i-k_i^{-1}}{q_i - q_i^{-1}}, \quad \forall i,j \in \hat{I},\\
		\end{split}
	\end{equation}
	and the quantum Serre relations
	\begin{equation}\label{eq:quantumserrerelations}
		\sum_{r=0}^{1-\widehat{C}_{i,j}} (-1)^r e_i^{(1-\widehat{C}_{i,j}-r)} e_j e_i^{(r)} = 0 = \sum_{r=0}^{1-\widehat{C}_{i,j}} (-1)^rf_i^{(1-\widehat{C}_{i,j}-r)} f_j f_i^{(r)},  \quad \forall i \neq j \in \hat{I},
	\end{equation}
	where $e_i^{(m)} = \frac{(e_i)^m}{[m]_{q_i}!}$ and $f_i^{(m)} = \frac{(f_i)^m}{[m]_{q_i}!}$ for $m \in \N$.

There is a unique $\hat{I}$-tuple of positive integers $(a_i)_{i \in \hat{I}}$ such that $a_{0} = 1$ and
\[\sum_{i \in \hat{I}} a_id_i\widehat{C}_{i,j} = 0 \text{ for all } j \in \hat{I}.\] 
These are determined by the multiplicity of simple roots in the longest root of $\g$: $\theta = \sum_{i=1}^n a_i \alpha_i$. Then $c := \prod_{i \in \hat{I}} k_i^{a_i}$ lies in the center of $\Uqghat$, and the quantum loop algebra $\Uqlg$ is defined to be the quotient of $\Uqghat$ by the ideal generated by $c-1$.

The algebras $\Uqghat$ and $\Uqlg$ can be defined as the affinization of finite type quantum groups through Drinfeld's realization.

\begin{theorem}[\cite{drinfeld1987new,beck1994braid}]  \label{thm:nontwisteddrinfeldrealizations}
	The non-twisted quantum loop algebra $\Uqlg$ is isomorphic to the algebra defined by generators $x_{i,k}^{\pm} (i \in I, k \in \mathbb{Z})$, $h_{i,k} (i \in I, k \in \mathbb{Z} \setminus \{0\})$, $k_i^{\pm1} (i \in I)$ and relations
	\begin{equation}\label{eq:nontwistedcommutationdrinfeld}
		\begin{split}
			&k_ik_i^{-1} = k_i^{-1}k_i = 1, \\
			&h_{i,k}h_{j,l}=h_{j,l}h_{i,k}, \; k_i h_{j,l} = h_{j,l}k_i, \; k_ik_j = k_j k_i,\\
			&k_ix_{j,k}^{\pm} = q_i^{\pm C_{i,j}} x_{j,k}^{\pm}k_i,\\
			&[h_{i,k}, x_{j,l}^{\pm}] = \pm \frac{1}{k}( [kC_{i,j}]_{q_i}) x_{j,k+l}^{\pm},\\
			&[x_{i,k}^+,x_{j,l}^-] =\delta_{i,j} (\frac{\phi_{i,k+l}^+ -  \phi_{i,k+l}^-}{q_i - q_i^{-1}}), \\
			&(u_1-q^{\pm C_{i,j}} u_2) x^{\pm}_i(u_1) x^{\pm}_j(u_2) = (q^{\pm C_{i,j}}u_1 - u_2) x^{\pm}_j(u_2) x^{\pm}_i(u_1),\quad \forall i,j \in I, \\
		\end{split}
	\end{equation}
	together with the quantum Drinfeld-Serre relations:
	\begin{equation}\label{eq:nontwistedquantumserredrinfeld}
		\sum_{\pi \in \mathfrak{S_m}} \sum_{r=0}^{m}(-1)^r \qbinom{1-C_{i,j}}{r}_{q_i} x_i^{\pm}(u_{\pi(1)}) \cdots x_i^{\pm}(u_{\pi(r)}) x_j^{\pm}(w) x_i^{\pm}(u_{\pi(r+1)}) \cdots x_i^{\pm}(u_{\pi(m)}) = 0, \; \forall i \neq j,
	\end{equation}
	
	where $m = 1- C_{i,j}$, $x^{\pm}_i(u) = \sum_{l \in \mathbb{Z}} x_{i,l}^{\pm} u^{-l}$, and the $\phi_{i,k}^{\pm}$ are defined by
	\[\phi_{i}^{\pm}(u) = \sum_{k=0}^{\infty} \phi_{i,\pm k}^{\pm} u^{\pm k} := k_i^{\pm 1} \exp(\pm (q_{i} - q_{i}^{-1}) \sum_{l=1}^{\infty} h_{i, \pm l}u^{\pm l}), \quad \forall i \in I.\]
\end{theorem}

We will be interested in the Borel subalgebra of $\Uqghat$, 

\begin{definition}
	The Borel subalgebra of $\Uqghat$ is the subalgebra generated by $e_i$, $k_i^{\pm}$, $\forall i \in \hat{I}$. We denote this subalgebra by $\Uqb$, and we call it the quantum affine Borel algebra.
\end{definition}

\begin{proposition}\cite[Theorem~4.21]{jantzen1996lectures}\label{prop:Borel def}
	 The quantum affine Borel algebra $\Uqb$ is isomorphic to the algebra defined by generators $e_i$, $k_i^{\pm}$, $i \in \hat{I}$ and relations
	 \[\begin{split}
	 	&k_i k_i^{-1} = k_i^{-1}k_i = 1, \; k_ik_j = k_j k_i, \; k_i e_j = q_i^{\widehat{C}_{i,j}}e_j k_i, \quad \forall i,j \in \hat{I},\\
	 	&\sum_{r=0}^{1-\widehat{C}_{i,j}} (-1)^r e_i^{(1-\widehat{C}_{i,j}-r)} e_j e_i^{(r)} = 0, \quad \forall i \neq j.
	 \end{split}\]
\end{proposition}

The Borel subalgebra of $\Uqghat$ contains Drinfeld generators $x^{+}_{i,m}$, $x^-_{i,k}$, $h_{i,k}$ ($i \in I, k \geq 1, m \geq 0$). In particular, it contains $\phi^+_{i,m}$ ($i \in I, m \geq 0$).

%----------------------subsection-------------------------%
\subsection{Representations and characters}\label{subsec:recall q char}
Let $\mathfrak{h}$ be the Cartan subalgebra of the finite-dimensional Lie algebra $\g$. Let $\omega_i \in \mathfrak{h}^*$ ($i \in I$) be the fundamental weights of $\g$, and let $\alpha_i \in \mathfrak{h}^*$ ($i \in I$) be the simple roots of $\g$. Recall that we have a Weyl group invariant inner product on $\mathfrak{h}^*$ such that $\langle \alpha_i, \omega_j \rangle = \delta_{i,j} d_i$. Denote by $P = \oplus_{i \in I} \Z \omega_i$ the weight lattice of $\g$.

For a finite-dimensional representation $V$ of $\Uqghat$, an element $\lambda \in \mathfrak{h}^*$ is called a weight of $V$ if the space
\[V_{\lambda} := \{v \in V | k_i v = q^{\langle \alpha_i,\lambda \rangle}v \}\]
is non zero. Then any weight $\lambda$ of $V$ is an element in the weight lattice $P$. Remark that when studying finite-dimensional representations, we are using the weights of $\g$ instead of those of $\hatg$.

As usual, we denote by $\Z[e^{\lambda}]_{\lambda \in P}$ the ring of finite sums $\sum_{\lambda \in P} c_{\lambda}e^{\lambda}$, where $c_{\lambda} \in \Z$ and the multiplication is given by $e^{\lambda_1} e^{\lambda_2} = e^{\lambda_1+\lambda_2}$. Recall that for any finite-dimensional representation $V$ of $\Uqghat$, its usual character is defined to be
\[\chi(V) = \sum_{\lambda \in P} \dim(V_{\lambda})e^{\lambda} \in \Z[e^{\lambda}]_{\lambda \in P}.\]

Then we recall the $q$-characters of finite-dimensional representations of $\Uqghat$.

For any $\Uqghat$-module $V$, a set of complex numbers $\bpsi : = (\psi_{i,\pm m}^{\pm})_{i \in I, m \in \N}$, is called an $l$-weight of $V$ if the space 
$\{v \in V | \phi_{i,\pm m}^{\pm}.v = \psi_{i,\pm m}^{\pm}v \}$ is non zero.

An $l$-highest weight representation of $\Uqghat$ is a $\Uqghat$-module $V$ generated by a vector $v$ such that 
	\[x_{i,k}^+.v = 0, \; \phi_{i,\pm m}^{\pm}.v = \psi_{i,\pm m}^{\pm}v, \quad \forall i \in I, k \in \Z, m \in \N, \]
for some complex numbers $\psi_{i,\pm m}^{\pm} \in \Ci$ ($i \in I, m \in \N$). We call such $\bpsi = (\psi_{i,\pm m}^{\pm})_{i \in I, m \in \N}$ a highest $l$-weight of $V$.

For any $l$-weight $\bpsi$, notice that $\phi_{i,0}^{+} = k_i$ are invertible, thus the constant terms $\psi_{i,0}^{+}$ are non-zero complex numbers. Let $(\lambda_i)_{i \in I} \in \Ci^I$ so that $\psi_{i,0} = q_i^{\lambda_i}$. We call $\sum_{i \in I} \lambda_i \omega_i \in \mathfrak{h}^*$ the weight of $\bpsi$, denoted by $\mathrm{wt}(\bpsi)$.

The finite-dimensional irreducible representations of $\Uqghat$ are classified by $l$-highest weight representations.

\begin{theorem}[\cite{chari1991quantum,chari1995quantum}]
	Any finite-dimensional irreducible representation of $\Uqghat$ is an $l$-highest weight representation. 
	
	Conversely, an $l$-highest weight representation is finite-dimensional if and only if its highest $l$-weight $\bpsi$ verifies:
	there exist polynomials $(P_i)_{i \in I} \in \mathbb{C}[u]^I$ such that $\forall i \in I$, $P_i(0)=1$ and
	
	\begin{equation}
		\sum_{m = 0}^{\infty} \psi_{i,m}^+ u^m = q_i^{\deg(P_i)} \frac{P_i(uq_i^{-1})}{P_i(uq_i)} = \sum_{m = 0}^{\infty} \psi_{i,-m}^- u^{-m},
	\end{equation}
	where we expand the rational function at $0$ and at $\infty$ respectively.
	
	We denote by $L(\bpsi)$ the unique irreducible $l$-highest weight representation of highest $l$-weight $\bpsi$.
\end{theorem}

Introducing the notation 
	\begin{equation}\label{eq:variable Y}
		Y_{i,a} = (1, \cdots, q_i\frac{1-uaq_i^{-1}}{1-uaq_i}, \cdots, 1) \in \Ci \llbracket u \rrbracket^I, \quad \forall i \in I, a \in \Ci^*
	\end{equation}
with the rational function at the $i$-th factor. Then by the above theorem, the highest $l$-weight $\bpsi$ of an irreducible finite-dimensional representation is a monomial in $Y_{i,a}$ ($i \in I, a \in \Ci^*$). Moreover, it is proved that any $l$-weight of a finite-dimensional representation of $\Uqghat$ can be written as a monomial in $Y_{i,a}^{\pm 1}$, $i \in I, a \in \Ci^*$ \cite{frenkel1999q}. 

The $q$-characters of finite-dimensional $\Uqghat$-modules were defined by E.~Frenkel and N.~Reshetikhin (and their Yangian analogues were defined by H.~Knight) \cite{frenkel1999q,knight1995spectra}. Let $\ringYa$ be the Laurent polynomial ring generated by formal symbols $Y_{i,a}^{\pm 1}$ ($i \in I, a \in \Ci^*$). For an $l$-weight $\bpsi$, denote by $[\bpsi]$ the corresponding monomial in $Y_{i,a}^{\pm 1}$ in $\ringYa$. For a finite-dimensional $\Uqghat$-module $V$, the $q$-character of $V$ is defined as
	\[\chi_q(V) = \sum_{\bpsi \in L} \dim(V_{\bpsi})[\bpsi] \in \ringYa,\]
where $L$ is the set of all $l$-weights of $V$, and
	\[V_{\bpsi} := \{v \in V \vert \exists p \geq 0, \forall i \in I,m \geq 0, (\phi_{i,\pm m}^{\pm} - \psi_{i,\pm m}^{\pm})^p.v = 0 \}.\]

In fact, the $q$-characters $\chi_q$ for finite-dimensional representations of $\Uqghat$ define a ring homomorphism from the Grothendieck ring $\mathrm{Rep}(\Uqghat)$ of the category of finite-dimensional representations to the commutative Laurent polynomial ring
\[\chi_q: \mathrm{Rep}(\Uqghat) \to \ringYa.\]

With the following monomials $A_{i,a}$ in $\ringYa$,
\begin{equation}\label{eq: A in terms of Y}
	A_{i,a} = Y_{i,aq_i^{-1}}Y_{i,aq_i}\prod_{j:C_{ji}= -1}Y^{- 1}_{j,a}  \prod_{j:C_{ji}= -2}Y^{- 1}_{j,aq^{-1}}Y^{- 1}_{j,aq}  \prod_{j:C_{ji}= -3}Y^{- 1}_{j,aq^{-2}}Y^{- 1}_{j,a}Y^{- 1}_{j,aq^2},
\end{equation}
it is proved that

\begin{theorem}[{\cite{frenkel1999q,frenkel2001combinatorics}}]
	Let $M$ be a monomial in $Y_{i,a}$, then every $l$-weight of the finite-dimensional $l$-highest weight representation $L(M)$ is a monomial in in $M\ringAa$.
	
	Furthermore, for a fixed non-zero complex number $a \in \Ci^*$, let $M$ be a monomial in variables $Y_{i,aq^n}$, $i \in I, n \in \Z$, and let $V$ be an $l$-highest weight module of highest $l$-weight $M$. Then the $q$-character of $V$ is an element in $M \ringA$.
\end{theorem}

When $M$ is a monomial in  $Y_{i,aq^n}$ ($i \in I, n \in \Z$), the normalized $q$-character of an $l$-highest weight representation $V$ of highest $l$-weight $M$ is defined to be 
\[\widetilde{\chi}_q (V) = \frac{\chi_q(V)}{M},\]
which is an element in the ring $\ringA$.

%----------------------subsection-------------------------%
\subsection{Braid groups and Weyl groups}
Let $W$ be the finite Weyl group associated to $\mathfrak{g}$ and $\mathcal{B}$ the corresponding braid group. Recall that $\mathcal{B}$ is the group generated by elements $T_i$, $i \in I$, with relations
\begin{align*}
	T_i T_j = T_j T_i, & \text{ if } C_{ji}C_{ij} = 0, \\
	T_i T_j T_i = T_j T_i T_j, & \text{ if } C_{ji}C_{ij} = 1, \\
	T_i T_j T_i T_j = T_j T_i T_j T_i, & \text{ if } C_{ji}C_{ij} = 2, \\
	(T_i T_j)^3 = (T_j T_i)^3, & \text{ if } C_{ji}C_{ij} = 3, \quad \forall i,j \in I. \\
\end{align*}

The Weyl group $W$ is generated by elements $s_i$, $i \in I$, with the same relations as $T_i$ above, together with the relations $s_i^2 = 1$, $\forall i \in I$.

Clearly, the map $T_i \mapsto s_i$, $\forall i \in I$, generates a group homomorphism from $\mathcal{B}$ to $W$. Therefore, any Weyl group action is naturally a braid group action.

There is a well-defined injective map from $W$ to $\mathcal{B}$, which associates an element $w \in W$ with the element $T_{i_1} \cdots T_{i_r}$ for a reduced expression $w = s_{i_1} \cdots s_{i_r}$, where $r = l(w)$ is the length of $w$. This element in $\mathcal{B}$ is independent of the choice of the reduced expression, and is denoted by $T_w$. Moreover, for two elements $w, w' \in W$, $T_w T_{w'} = T_{ww'}$ if and only if $l(w) + l(w') = l(ww')$.

Let $\omega_i^{\vee} \in \mathfrak{h}$ be fundamental coweights of $\g$, and $P^{\vee} = \oplus_{i \in I} \Z \omega_i^{\vee}$ the coweight lattice. Let $\widehat{W} := W \ltimes P^{\vee}$ be the extended affine Weyl group and $\widehat{\mathcal{B}}$ be the corresponding extended affine braid group. Recall that $\widehat{W}$ is generated by simple reflections $s_i$, $i \in \hat{I}$, together with diagram automorphisms $\tau$ \cite[Chapter~VI]{bourbaki2002lieIV}.

%----------------------------------------------
%SECTION
%----------------------------------------------
\section{Limits of \texorpdfstring{$q$}{}-characters}\label{sec:limits of q characters}
In this section, we begin with a review of the braid group actions on $l$-weights. Then, we define various topological completions of the ring of $q$-characters. Finally, we explore a phenomenon of convergence of $q$-characters.

%----------------------subsection-------------------------%

\subsection{Braid group actions on \texorpdfstring{$l$}{}-weights}\label{subsec:Chari braid actions}
In \cite{bouwknegt1998deformed,chari2002braid}, a $\mathcal{B}$-action on the commutative ring $\ringYa$ is defined. 

 \begin{definition}\label{def:braid action on ringY}
	The braid group $\mathcal{B}$ acts on the ring $\ringYa$ by
	\begin{equation}\label{eq:B action on Y}
		\begin{split}
			& T_i Y^{\pm 1}_{j,a} = Y^{\pm 1}_{j,a},\text{ if } i \neq j,\\
			& T_i Y^{\pm 1}_{i,a} = Y^{\mp 1}_{i,aq_i^2} \prod_{j:C_{ji}= -1}Y^{\pm 1}_{j,aq_i}  \prod_{j:C_{ji}= -2}Y^{\pm 1}_{j,aq}Y^{\pm 1}_{j,aq^3}  \prod_{j:C_{ji}= -3}Y^{\pm 1}_{j,aq}Y^{\pm 1}_{j,aq^3}Y^{\pm 1}_{j,aq^5}.
		\end{split}
	\end{equation}
\end{definition}

With the notation $A_{i,a}$ in \eqref{eq: A in terms of Y}, the braid group action \eqref{eq:B action on Y} can be written as 
\begin{equation}
	\begin{split}
		& T_i Y^{\pm 1}_{j,a} = Y^{\pm 1}_{j,a}, \text{ if } i \neq j,\\
		& T_i Y^{\pm 1}_{i,a} = Y^{\pm 1}_{i,a} A^{\mp 1}_{i,aq_i}.
	\end{split}
\end{equation}

Moreover, the $\mathcal{B}$-action preserves the subring $\ringAa$ of $\ringYa$, and the action is given by: $\forall i, j \in I$,
\begin{equation}
	\begin{split}
		& T_i A^{\pm 1}_{j,a} = A^{\pm 1}_{j,a}, \text{ if } C_{ij} = 0, \\
		& T_i A^{\pm 1}_{j,a} = A^{\pm 1}_{j,a} A^{\pm 1}_{i,aq_i}, \text{ if } C_{ij} = -1, \\
		& T_i A^{\pm 1}_{j,a} = A^{\pm 1}_{j,a} A^{\pm 1}_{i,aq} A^{\pm 1}_{i,aq^3}, \text{ if } C_{ij} = -2, \\
		& T_i A^{\pm 1}_{j,a} = A^{\pm 1}_{j,a} A^{\pm 1}_{i,aq} A^{\pm 1}_{i,aq^3} A^{\pm 1}_{i,aq^5}, \text{ if } C_{ij} = -3, \\
		& T_i A^{\pm 1}_{i,a} = A^{\mp 1}_{i,aq_i^2}.
	\end{split}
\end{equation}

Let $\mathfrak{r}$ be the multiplicative group of $I$-tuple of rational functions which are regular and non-zero at $0$. This braid group action is extended by E.~Frenkel and D.~Hernandez to any $l$-weights $\bpsi \in \mathfrak{r}$ \cite[Definition~3.4]{frenkel2023extended}. Let us recall the definition here. 

\begin{definition}\label{def:barid group action on psi}
	The braid group $\mathcal{B}$ acts as group isomorphisms on $\mathfrak{r}$: $\forall i \in I$, $a \in \Ci^*$, and $\lambda \in \mathfrak{h}^* \simeq (\Ci^*)^I$
	\begin{equation}
		\begin{split}
			&T_i \bpsi_{j,a}= \bpsi_{j,a}, \text{ if } i \neq j,\\
			&T_i \bpsi_{i,a}= \bpsi_{i,aq_i^2}^{-1}	\prod_{j:C_{ij}= -1} \bpsi_{j,aq_i}  \prod_{j:C_{ij}= -2} \bpsi_{j,a}\bpsi_{j,aq^2}  \prod_{j:C_{ij}= -3} \bpsi_{j,aq^{-1}}\bpsi_{j,aq}\bpsi_{j,aq^3},\\
			&T_i \lambda = s_i \lambda, 
		\end{split}
	\end{equation}
	where
	\begin{equation}\label{eq:psi pm}
		\bpsi_{i,a}^{\pm 1} = (1,\cdots,(1-au)^{\pm 1},\cdots,1)
	\end{equation}
	with the rational function at the $i$-th factor.
\end{definition}

We note that this defines a braid group action, but not a Weyl group action.

By formula \eqref{eq:variable Y}, each monomial in $\ringYa$ is identified with an element in $\mathfrak{r}$. The above braid group action preserves the set of monomials in $\ringYa$, and the action coincides with the formula \eqref{eq:B action on Y}.

\begin{remark}
	One should pay attention to the index of Cartan matrices in the above actions. $C_{ij}$ and $C_{ji}$ are different in non-simply-laced cases.
\end{remark}

\begin{remark}
	We remark that the formulae written in \cite[Definition~3.4]{frenkel2023extended} are in fact the inverse of $T_i$ in Definition~\ref{def:barid group action on psi}, which are given by
	\[\begin{split}
		&T_i^{-1} \bpsi_{j,a}= \bpsi_{j,a}, \text{ if } i \neq j,\\
		&T_i^{-1} \bpsi_{i,a}= \bpsi_{i,aq_i^{-2}}^{-1}	\prod_{j:C_{ij}= -1} \bpsi_{j,aq_i^{-1}}  \prod_{j:C_{ij}= -2} \bpsi_{j,a}\bpsi_{j,aq^{-2}}  \prod_{j:C_{ij}= -3} \bpsi_{j,aq}\bpsi_{j,aq^{-1}}\bpsi_{j,aq^{-3}}.
	\end{split}
	\]
\end{remark}

The following theorem of Chari states the relation between the braid group action on $\mathfrak{r}$ and $q$-characters.

    \begin{theorem}[{\cite[Proposition~4.1]{chari2002braid}}]\label{thm:Chari l weight}
        Let $M$ be a monomial in $Y_{i,a}$, $i \in I, a \in \Ci^*$, and let $L(M)$ be the simple finite-dimensional $l$-highest weight representation of highest $l$-weight $M$. Then $\forall w \in W$, $T_wM$ appears as a monomial in $\chi_q(L(M))$. 
    
        More precisely, if $\lambda \in \mathfrak{h}^*$ is the weight of $M$, then $w\lambda$ are also weights in $L(M)$, and the weight spaces $L(M)_{w\lambda}$ of $L(M)$ of weight $w\lambda$ are one-dimensional $\forall w \in W$. In particular, vectors in $L(M)_{w\lambda}$ are eigenvectors of $\phi_i^{\pm}(u) \in \Uqhaffine$. The associated eigenvalues of vectors in $L(M)_{w\lambda}$ are given by $T_wM \in \mathbb{C}\llbracket u \rrbracket^{I}$.
    \end{theorem}
        
%----------------------subsection--------------------%
\subsection{Ring of \texorpdfstring{$q$}{}-characters}
In this section, we prepare necessary backgrounds for studying convergent phenomenon of $q$-characters. We study rings where the $q$-characters land on. These rings are equipped with a topology induced from their weights, and we explore their topological completions.

Recall that $\alpha_i$ ($i \in I$) are the simple roots of $\g$. Let 
\[\Lambda = \oplus_{i \in I} \Z \alpha_i\] be the root lattice of $\g$, and 
\[\Lambda_- = \oplus_{i \in I} \N (-\alpha_i)\] be the cone of negative roots. For each element $w \in W$, its action on the root lattice provides a cone in another direction \[w\Lambda_- := \oplus_{i \in I} \N w(-\alpha_i).\] Clearly $w\Lambda_-$ is a monoid under addition with an ideal $w\Lambda_- \setminus \{0\}$.

\subsubsection{}
We begin with the ring of usual characters. 

We denote the Laurent polynomial ring
\[\mathscr{C} = \Z [e^{\pm \alpha_i}]_{i \in I}.\] 

It is $\Lambda$-graded by $\deg(e^{\pm \alpha_i}) = \pm \alpha_i$. Consider the subring $\mathscr{C}_w$ of $\mathscr{C}$ generated by homogeneous elements of $\Lambda$-degree in $w\Lambda_-$.

The $w\Lambda_-$-graded ring $\mathscr{C}_w$ has an ideal which is generated by homogeneous elements of degree in $w\Lambda_- \setminus \{0\}$. $\mathscr{C}_w$ is equipped with a topology induced by this ideal. 

\begin{definition}
	Let $\overline{\mathscr{C}}_w$ be the topological completion of $\mathscr{C}_w$. 
	
	An element in the complete topological ring $\overline{\mathscr{C}}_{w}$ is a (possibly) infinite sum 
	\[\sum_{\lambda \in w\Lambda_-} c_{\lambda}e^{\lambda}, \quad c_{\lambda} \in \mathbb{Z}.\] 
\end{definition}

Clearly, This completion can be described as the ring of formal power series in $e^{-w(\alpha_i)}$:
    \[\overline{\mathscr{C}}_w \simeq \Z \llbracket e^{-w(\alpha_i)} \rrbracket_{i \in I}. \]

\subsubsection{}
Then we deal with the ring of $q$-characters. Fix a non-zero complex number $a \in \Ci^*$.

Recall that when $M$ is a monomial in $Y_{i,aq^n}$ ($i \in I, n \in \Z$), the normalized $q$-character of $L(M)$ is an element in the Laurent polynomial ring 
 \[\mathscr{A} = \ringA.\] 
This ring is $\Lambda$-graded by their weights
    \begin{equation}\label{eq:Lambda degree}
    	\deg(A_{i,aq^n}^{\pm 1}) = \pm \alpha_i.
    \end{equation}

Analogous to the usual characters, for each element $w \in W$, we consider the subring $\mathscr{A}_w$ of $\mathscr{A}$ generated by homogeneous elements of $\Lambda$-degree in $w\Lambda_-$. A topology is defined on the $w\Lambda_-$-graded ring $\mathscr{A}_w$ by the ideal generated by homogeneous elements of degree in  $w\Lambda_- \setminus \{0\}$. 

\begin{definition}
	Let $\overline{\mathscr{A}}_w$ be its topological completion.
	
	An element in $\overline{\mathscr{A}}_w$ is a (possibly) infinite sum of monomials in $A_{i,aq^n}^{\pm 1}$ whose $\Lambda$-degree lies $w\Lambda_-$, and for each given $\alpha \in w\Lambda_-$, there are only finitely many terms of degree $\alpha$.
\end{definition}

\begin{remark}
This definition is related to Frenkel and Hernandez's definitions in \cite[Section~2.3]{frenkel2022weyl}. We remark that this ring $\overline{\mathscr{A}}_w$ is smaller than the ring $\widetilde{\mathscr{Y}}^w$ in \cite{frenkel2022weyl}, as we restricted to the subring consisting of $A_{i,a}^{\pm 1}$. Moreover, homogeneous elements in $\overline{\mathscr{A}}_w$ have degree in $w\Lambda_-$, while homogeneous elements in the ring $\widetilde{\mathscr{Y}}^w$ in \cite{frenkel2022weyl} may have degree in the weight lattice $P$.
	
On the other hand, we remark that this topological completion is larger than one expected. For example, when $w = e$ the neutral element, the topological completion $\overline{\mathscr{A}}_e$ is larger than the ring $\Z \llbracket A_{i,aq^n}^{-1} \rrbracket_{i \in I,n \in \Z}$ considered by Frenkel-Mukhin and Hernandez-Jimbo \cite{frenkel2001combinatorics,hernandez2012asymptotic}. In fact, the $0$-degree space of $\overline{\mathscr{A}}_e$ is infinite dimensional which contains elements of the form $A_{i,aq^m} A_{i,aq^n}^{-1}$, while the $0$-degree space of $\Z \llbracket A_{i,aq^n}^{-1} \rrbracket_{i \in I,n \in \Z}$ is one-dimensional.
\end{remark}

\subsubsection{}
We also need the mixed ring containing both $A_{i,a}$ and $e^{\alpha}$. Similarly, for a fixed $a \in \Ci^*$, we consider the $\Lambda$-graded ring 
\[\mathscr{CA} := \Z [e^{\pm \alpha_i}, A_{i,aq^n}^{\pm 1}]_{i \in I, n \in \Z} \simeq \mathscr{C} \otimes \mathscr{A}.\] 

Let $\mathscr{CA}_w$ be its subring generated by homogeneous elements of degree in $w\Lambda_-$, and its topological completion is denoted by $\overline{\mathscr{CA}}_w$.

%----------------------subsection--------------------%
\subsection{Normalization of \texorpdfstring{$q$}{}-characters}
We define various normalization of $q$-characters corresponding to Weyl group elements, which generalize the normalized $q$-character $\widetilde{\chi}_q$.

Throughout this section, $a \in \Ci^*$ will be a fixed non-zero complex number. 

    \begin{definition}
    	Let $M$ be a monomial in $Y_{i,aq^n}$, $i \in I, n \in \Z$.
        For each $w \in W$, recall that $T_wM$ is a monomial in $\ringY$ as defined in Definition~\ref{def:braid action on ringY}. Recall also that $\chi_q(L(M))$ is an element in $\ringY$. We define the $w$-normalized $q$-character of $L(M)$ to be
        \begin{equation}
            \chi_q^w(L(M)) := \frac{\chi_q(L(M))}{T_wM} \in \ringY.
        \end{equation}
    \end{definition}
In particular, when $w = e$ the neutral element, this is what we called the normalized $q$-character, $\chi_q^e(L(M)) = \widetilde{\chi}_q(L(M))$.

Throughout this section, we fix an index $i \in I$, and we consider the highest $l$-weight of KR-modules
    \begin{equation}\label{eqn:KR weight}
        M_k = M^{(i)}_k := Y_{i,aq_i^{-2k+1}} \cdots Y_{i,aq_i^{-3}}Y_{i,aq_i^{-1}}, \quad \forall k \in \N^*.
    \end{equation}
Let $m$ be a monomial in $Y_{j,b}$ ($j \in I , b \in \Ci^*$) which is independent of $k$.

In general, if we decompose $m = \prod_{b \in \Ci/q^{\Z}} m_b$ according to the subscripts such that $m_b \in \Z[Y_{i,bq^n}]_{i \in I, n \in \Z}$, then $L(M_km)$ can be decompose into an irreducible tensor product, see \cite[Section~4.8]{chari1991quantum} for $\g = \sltwo$ and \cite[Theorem~5.1]{chari2010beyond} for general types:
\[L(M_km) = L(Y_{i,aq_i^{-2k+1}} \cdots Y_{i,aq_i^{-3}}Y_{i,aq_i^{-1}}m_a) \otimes L(m_a^{-1} m),\]
where the first factor $Y_{i,aq_i^{-2k+1}} \cdots Y_{i,aq_i^{-3}}Y_{i,aq_i^{-1}}m_a \in \Z[Y_{i,aq^n}]_{i \in I, n \in \Z}$ and the second factor $m_a^{-1}m$ is independent of $k$. Thus it is enough to study the case $m \in \Z[Y_{i,aq^n}]_{i \in I, n \in \Z}$.

From now on, we assume $m$ to be a monomial in $\Z[Y_{i,aq^n}]_{i \in I, n \in \Z}$ for simplicity.

    \begin{lemma}\label{lemma:chiqw in Aw}
        $\chi_q^{w}(L(M_km))$ is an element in the ring $\mathscr{A}_w$.
    \end{lemma}
    \begin{proof}
        Recall that the $q$-character $\chi_q(L(M_km))$ lies in $M_km \ringA$. Furthermore, by Theorem~\ref{thm:Chari l weight}, $T_w(M_km)$ appears as a monomial in $\chi_q(L(M_km))$, thus $\frac{M_km}{T_w(M_km)}$ is a monomial in $\ringA$. Therefore, the $w$-normalized $q$-character 
        \[\chi_q^{w}(L(M_km)) = \frac{\chi_q(L(M_km))}{T_w(M_km)} = \frac{M_km}{T_w(M_km)} \widetilde{\chi}_q(L(M_km))\] lies in the ring $\Z [A_{i,aq^n}^{\pm 1}]_{i \in I, n \in \Z}$. 
        
        Moreover, let $\lambda$ be the weight of $M_km$, then the set of weights of $L(M_km)$ are contained in $\lambda + \Lambda_-$ and is stable under the finite Weyl group $W$ action. Thus the set of weights in $L(M_km)$ are contained in $w(\lambda) + w\Lambda_-$, $\forall w \in W$. Therefore $\chi_q^{w}(L(M_km)) \in \mathscr{A}_w$.
    \end{proof}

%-----------------------------subsection------------------------------%

\subsection{Projected \texorpdfstring{$q$}{}--characters and their limits}
Throughout this section, $M_k$ is given by \eqref{eqn:KR weight} with a fixed index $i \in I$ and a fixed non-zero complex number $a \in \Ci^*$, and $m$ is a monomial in $Y_{j,aq^n}$ ($j \in I, n \in \Z$) which is independent of $k$. In this section, we study the convergent phenomenon of $q$-characters of $l$-highest weight modules $L(M_km)$. 

Recall that when $w = e$, the normalized $q$-characters $\chi_q^e(L(M_km)) = \widetilde{\chi}_q(L(M_km))$ form a convergent sequence in the topological ring $\Z \llbracket  A_{i,aq^n}^{-1} \rrbracket_{i \in I, n \in \Z}$. When $m = 1$, this is known as the $T$-systems \cite{nakajima2003t,hernandez2006kirillov,hernandez2012asymptotic}. For general $m$, the proof is similar \cite[Theorem~7.6]{hernandez2016clusterO}. 

For general $w \in W$, the $w$-normalized $q$-characters $\chi_q^w(L(M_km))$ are no longer convergent. 

	\begin{example}
		When $\g = \sltwo$ and the monomial $m = 1$. Under the $s_1$-normalization, we have
		\[\chi_q^{s_1}(L(M_k)) = 1 + A_{aq^{-2k+2}} + A_{aq^{-2k+2}}A_{aq^{-2k+4}} + \cdots + A_{aq^{-2k+2}}\cdots A_{a}, \; \forall k \in \N^*.\]
		These form a sequence in $\mathscr{A}_{s_1}$, which does not converge in the ring $\overline{\mathscr{A}}_{s_1}$.
	\end{example}

However, we define the notion of projected $w$-normalized $q$-characters, and we conjecture that the projected $w$-normalized $q$-characters converge.

    \begin{definition}
        Let $R \in \Z$ be an integer. Recall that by definition $ \mathscr{A} = \ringA$ and $\mathscr{CA} = \Z [e^{\pm \alpha_i}, A_{i,aq^n}^{\pm 1}]_{i \in I, n \in \Z}$. Define a ring homomorphism 
        \begin{equation}\label{eq:piR}
        \pi_R : \mathscr{A} \to \mathscr{CA}
        \end{equation}
        which maps $A_{i,aq^n}^{\pm 1}$ to $e^{\pm \alpha_i}$ when $n < R$, and maps $A_{i,aq^n}^{\pm 1}$ to $A_{i,aq^n}^{\pm 1}$ when $n \geq R$. In particular, $\pi_R$ preserves the $\Lambda$-degree \eqref{eq:Lambda degree}.
    \end{definition}

    \begin{definition}
        Let $R \in \Z$ be an integer. We define the projected $w$-normalized $q$-character $\pi_q^{w,\geq R}(L(M_km))$ to be $\pi_R(\chi_q^w(L(M_km))) \in \mathscr{CA}$. This is well-defined by Lemma~\ref{lemma:chiqw in Aw}, moreover, the projected $w$-normalized $q$-character
        \[\pi_q^{w,\geq R}(L(M_km)) \in \mathscr{CA}_w \subset \overline{\mathscr{CA}}_w.\]
    \end{definition}

    \begin{conjecture}\label{conj:limit of projected w norm}
    	For fixed $i \in I$, $a \in \Ci^*$, and a fixed monomial $m \in \Z[Y_{j,aq^n}]_{j \in I, n \in \Z}$. Let $M_k$ be as in \eqref{eqn:KR weight}, $\forall k \in \N^*$. 
        \begin{enumerate}
            \item The sequence $\pi_q^{w,\geq R}(L(M_km))$, $k = 1,2,\cdots$, is convergent in the topological ring $\overline{\mathscr{CA}}_w$ when $k \to + \infty$. We denote the limit of this sequence by $\pi_{q,\infty}^{w,\geq R}(m)$.
            \item  Moreover, the sequence $\pi_{q,\infty}^{w,\geq R}(m)$, $R= 0,-1,-2,\cdots$, is convergent in the ring $\overline{\mathscr{CA}}_w$ when $R \to -\infty$. We denote its limit by $\pi_{q,\infty}^{w}(m)$.
            \end{enumerate}
    \end{conjecture}

	\begin{definition}
		The limit $\pi_{q,\infty}^{w}(m)$ in the above conjecture is called the projected limit. When the given monomial $m = 1$, we simply write $\pi_{q,\infty}^{w,\geq R}(m)$ as $\pi_{q,\infty}^{w,\geq R}$, and write the projected limit $\pi_{q,\infty}^{w}(m)$ as $\pi_{q,\infty}^{w}$.
	\end{definition}

When $\g = \mathfrak{sl}_n$, we have explicit formula on $q$-characters of KR-modules \cite{feigin2017finite}, which allows us to verify this conjecture in some cases.
For example, we verify this conjecture for $\g = \sltwo$ and $\g = \mathfrak{sl}_3$ for KR-modules.

    \begin{example}\label{ex:projectedlimitsl2}
        For $\g = \sltwo$, we give an example calculated in details.
        \begin{itemize}
            \item 
        When $w = e$, the normalized $q$-character is
            \[\chi_q^e(L(M_k)) = 1 + A^{-1}_{a} +  A^{-1}_{a}A_{aq^{-2}}^{-1} + \cdots + A^{-1}_{a}\cdots A^{-1}_{aq^{-2k+2}},\]
        thus the projected normalized $q$-character at $R = -2N+2$ or $R = -2N+1$ is
            \[\begin{split}
                \pi_q^{e,\geq R}(L(M_k)) =& 1 + A^{-1}_{a} +\cdots + A^{-1}_{a}\cdots A^{-1}_{aq^{-2N+4}}\\ 
                &+ A^{-1}_{a}\cdots A^{-1}_{aq^{-2N+2}}(1 + e^{-\alpha} + \cdots + e^{-(k-N)\alpha}), \; \forall k > N,
                \end{split}\]
            which converges in the topological ring $\overline{\mathscr{CA}}_e$ as $k \to \infty$, and the limit is
            \[\pi_{q,\infty}^{e,\geq R} = 1 + A^{-1}_{a} +\cdots + A^{-1}_{a}\cdots A^{-1}_{aq^{-2N+4}} + A^{-1}_{a}\cdots A^{-1}_{aq^{-2N+2}} \frac{1}{1-e^{-\alpha}}. \] 
        This sequence again converges in $\overline{\mathscr{CA}}_e$ when $R \to -\infty$, and the limit is 
            \[\pi_{q,\infty}^{e} = 1 + A^{-1}_{a} +  A^{-1}_{a}A_{aq^{-2}}^{-1} +\cdots + A^{-1}_{a}\cdots A^{-1}_{aq^{-2n+2}} + \cdots.\]
            \item
        When $w = s_1$, the $s_1$-normalized $q$-character is
                \[\chi_q^{s_1}(L(M_k)) = 1 + A_{aq^{-2k+2}} + A_{aq^{-2k+2}}A_{aq^{-2k+4}} + \cdots + A_{aq^{-2k+2}}\cdots A_{a},\]
        thus the projected $s_1$-normalized $q$-character at  $R = -2N+2$ or $R = -2N+1$ is
            \[\begin{split}
                 \pi_q^{s_1,\geq R}(L(M_k)) & = 1 + e^{\alpha} + \cdots + e^{(k-N)\alpha}\\
                &+ e^{(k-N)\alpha}(1 + A_{aq^{-2N+2}} + A_{aq^{-2N+2}}A_{aq^{-2N+4}} + \cdots + A_{aq^{-2N+2}}\cdots A_{a}),
            \end{split}\] 
            for $k > N$.
        As $k \to \infty$, the sums above have a limit in $\overline{\mathscr{CA}}_{s_1}$ which equals to 
            \[\pi_{q,\infty}^{s_1,\geq R} = 1 + e^{\alpha} + \cdots + e^{n\alpha} + \cdots = \frac{1}{1-e^{\alpha}}.\] 
        They are in fact independent of $R$, thus when $R \to -\infty$, we have the projected limit
            \[\pi_{q,\infty}^{s_1} =\frac{1}{1-e^{\alpha}}.\]
        \end{itemize}
    \end{example}

    \begin{example}\label{ex:projectedlimitsl3}
        When $\g = \mathfrak{sl}_3$, it is sufficient to calculate at the index $i = 1$. The $q$-characters of KR-modules $L(M_k^{(1)}) = L(Y_{1,aq^{-2k+1}}\cdots Y_{1,aq^{-3}}Y_{1,aq^{-1}})$ are also explicit.
        \begin{itemize}
            \item When $w=e$, the normalized $q$-character 
                \begin{equation}
                    \chi_q^{e}(L(M_k^{(1)})) = \frac{\chi_q(L(M_k^{(1)}))}{M_k^{(1)}} = \sum_{-1 \leq m \leq n \leq k-1} \prod_{l=0}^{n} A_{1,aq^{-2l}}^{-1} \prod_{s=0}^{m} A_{2,aq^{-2s+1}}^{-1}.
                \end{equation}
            Similar to $\sltwo$ case, we can calculate
                \[\pi_{q,\infty}^{s_1, \geq -2N} = \sum_{-1 \leq m \leq n} (\prod_{l=0}^{\min(n,N)} A_{1,aq^{-2l}}^{-1}e^{-(n-\min(n,N))\alpha_1}) (\prod_{s=0}^{\min(m,N)} A_{2,aq^{-2s+1}}^{-1}e^{-(m-\min(m,N))\alpha_2}).\]
                \[\pi_{q,\infty}^{s_1, \geq -2N+1} = \sum_{-1 \leq m \leq n} (\prod_{l=0}^{\min(n,N-1)} A_{1,aq^{-2l}}^{-1}e^{-(n-\min(n,N-1))\alpha_1}) (\prod_{s=0}^{\min(m,N)} A_{2,aq^{-2s+1}}^{-1}e^{-(m-\min(m,N))\alpha_2}).\] 
            Furthermore, they converge in $\overline{\mathscr{CA}}_{s_1}$ when $N \to \infty$, and the limit is 
                \[\pi_{q,\infty}^{e} = \sum_{-1 \leq m \leq n} (\prod_{l=0}^{n} A_{1,aq^{-2l}}^{-1}) (\prod_{s=0}^{m} A_{2,aq^{-2s+1}}^{-1}).\]

            \item When $w=s_1$, $T_1M_k^{(1)} = Y^{-1}_{1,aq^{-2k+3}}\cdots Y^{-1}_{1,aq^{-1}}Y^{-1}_{1,aq}Y_{2,aq^{-2k+2}}\cdots Y_{2,aq^{-2}}Y_{2,a}$, and we calculate $s_1$-normalized $q$-character
                \[\chi_q^{s_1}(L(M_k^{(1)})) = \frac{\chi_q(L(M_k^{(1)}))}{T_1M_k^{(1)}} = \sum_{\substack{m,n \geq -1\\ m + n \leq k-2}} \prod_{l=0}^{n} A_{1,aq^{-2k+2+2l}} \prod_{s=0}^{m} A_{2,aq^{-2s+1}}^{-1}.\]
            Therefore the projected $s_1$-normalized $q$-character at $R = -2N$ or $R = -2N+1$ has a limit when $k \to \infty$:
                \[\pi_{q,\infty}^{s_1,\geq R} =\sum_{m,n \geq -1} e^{(n+1)\alpha_1} (\prod_{s=0}^{\min(m,N)} A_{2,aq^{-2s+1}}^{-1}e^{-(m-\min(m,N))\alpha_2}).\]
            Again they converge as $R \to -\infty$ and the limit is 
                \[\pi_{q,\infty}^{s_1}= (1 + e^{\alpha_1} + e^{2\alpha_1} + \cdots)(\sum_{m=-1}^{\infty} \prod_{s=0}^{m} A_{2,aq^{-2s+1}}^{-1}).\]

            \item When $w = s_2s_1$, we can also calculate similarly to the $\sltwo$ case that
            \[\pi_{q,\infty}^{s_2s_1}= \sum_{M = 0}^{\infty}\sum_{N = 0}^{M} e^{N\alpha_1 + M \alpha_2} = \frac{1}{1-e^{\alpha_2}}\frac{1}{1-e^{\alpha_1 + \alpha_2}}.\]
        \end{itemize}
        These are all possible $w$-normalization in this example, since the other normalization are $\pi_{q,\infty}^{s_2} = \pi_{q,\infty}^{e}$, $\pi_{q,\infty}^{s_1s_2} =\pi_{q,\infty}^{s_1}$ and $\pi_{q,\infty}^{s_1s_2s_1} = \pi_{q,\infty}^{s_2s_1}$. Thus we have verified the conjecture when $\g = \mathfrak{sl}_3$ and $m = 1$.
    \end{example}

    \begin{remark}
        According to the explicit formula of $q$-characters in \cite[Equation~4.4]{feigin2017finite}, we can also easily verify Conjecture~\ref{conj:limit of projected w norm} for KR-modules of type $\mathfrak{sl}_n$ and when $w = s_i$ is a simple reflection.

        The conjecture is also verified for some representations $L(M_km)$ other than KR-modules. One such example will be shown in Example~\ref{ex:more than KR case}.
    \end{remark}
    
    \begin{remark}
    	In these two examples, we see that when $w = w_0$ the longest element, the projected limits are in fact contained in $\overline{\mathscr{C}}_{w_0}$. This fact is related to \cite[Equation~4.7]{feigin2017finite}, and we will prove it at the end (Proposition~\ref{prop:longest w0}).
    \end{remark}

%------------------------------------------------
%SECTION
%------------------------------------------------

\section{Inductive systems}\label{sec:inductive systems}
It is known that when $w = e$, the limit of normalized $q$-character of KR-modules has an algebraic interpretation: it is given by the $q$-character of prefundamental representations of Borel subalgebras \cite{hernandez2012asymptotic}. In fact, these prefundamental representations were constructed from inductive systems of finite-dimensional representations.

In this section, we construct the inductive systems and projective systems twisted by a Weyl group element. We hope that these inductive limits can explain the convergent phenomenon in Conjecture~\ref{conj:limit of projected w norm}.

%--------------------------subsection-----------------------------%
\subsection{Lusztig's automorphisms}
We begin by reviewing Lusztig's automorphisms and their connection to the braid group actions on $l$-weights.

\subsubsection{Definitions}
Recall that $\widehat{\mathcal{B}}$ is the extended affine braid group. Lusztig have defined four different $\widehat{\mathcal{B}}$-actions on $\Uqghat$ \cite{lusztig1993introduction}. These automorphisms are denoted by $T_{i,\pm 1}', T_{i,\pm 1}''$ by Lusztig. In \cite{beck1994braid}, the following braid group actions on quantum affine algebras are used to define root vectors and are used to express the Drinfeld generators in terms of Drinfeld-Jimbo generators. This braid group action in \cite{beck1994braid} corresponds to $T_{i,-1}''$ in the notation of Lusztig. In this article, we denote these Lusztig automorphisms by $T_i'' = T_{i,-1}''$ for simplicity.

Recall that the $T_i''$ actions are defined by
    \begin{equation}
        \begin{aligned}
            & T''_i e_i=-f_i k_i, \quad T''_i e_j=\sum_{s=0}^{-\widehat{C}_{i j}}(-1)^{s-\widehat{C}_{i j}} q_i^{-s} e_i^{(-\widehat{C}_{i j}-s)} e_j e_i^{(s)} \quad \text { if } i \neq j \in \hat{I}, \\
            & T''_i f_i=-k_i^{-1} e_i, \quad T''_i f_j=\sum_{s=0}^{-\widehat{C}_{i j}}(-1)^{s-\widehat{C}_{i j}} q_i^s f_i^{(s)} f_j f_i^{(-\widehat{C}_{i j}-s)} \quad \text { if } i \neq j \in \hat{I}, \\
            & T''_i k_\beta=k_{s_i \beta}, \quad \forall \beta \in \Lambda_{\mathrm{aff}} :=  \oplus_{i \in \hat{I}} \Z \alpha_i,
        \end{aligned}
    \end{equation}
    together with diagram automorphisms $\tau$ of affine Dynkin diagrams: 
    \[T''_{\tau}(e_i) = e_{\tau(i)}, T''_{\tau}(f_i) = f_{\tau(i)}, T''_{\tau}(k_i) = k_{\tau(i)}. \]

In this article, we need another Lusztig automorphism $T_{i,-1}'$ which is defined as follows. We use simply the symbol $T_i = T_{i,-1}'$ for these automorphisms throughout this article.

    \begin{equation}\label{eq:braidgroupactions}
        \begin{aligned}
            & T_i e_i=-k_i^{-1}f_i, \quad T_i e_j=\sum_{s=0}^{-\widehat{C}_{i j}}(-1)^{(-\widehat{C}_{i j}-s)} q_i^{(\widehat{C}_{i j}+s)} e_i^{(-\widehat{C}_{i j}-s)} e_j e_i^{(s)} \quad \text { if } i \neq j \in \hat{I}, \\
            & T_i f_i=-e_i k_i, \quad T_i f_j=\sum_{s=0}^{-\widehat{C}_{i j}}(-1)^{(-\widehat{C}_{i j}-s)} q_i^{(-\widehat{C}_{i j}-s)} f_i^{(s)} f_j f_i^{(-\widehat{C}_{i j}-s)} \quad \text { if } i \neq j \in \hat{I}, \\
            & T_i k_\beta=k_{s_i \beta}, \quad \forall \beta \in \Lambda_{\mathrm{aff}} :=  \oplus_{i \in \hat{I}} \Z \alpha_i,\\
            &T_{\tau}(e_i) = e_{\tau(i)}, T_{\tau}(f_i) = f_{\tau(i)}, T_{\tau}(k_i) = k_{\tau(i)}.
        \end{aligned}
    \end{equation}

    We use the same notation $T_i$ as for the braid group actions on $l$-weights without risk, since Lusztig's automorphisms act on the algebra $\Uqghat$ while the braid group actions in Definition~\ref{def:barid group action on psi} is defined on $I$-tuple of functions.
    
%--------------------------subsection------------------------------%
\subsubsection{Relation with symmetries} \label{subsec:relationwithsymmetry}
Recall that any finite-dimensional $\Uqghat$-module $V$ has a weight space decomposition $V = \oplus_{\lambda \in \mathfrak{h}^*} V_{\lambda}$, where the weight spaces $V_{\lambda} := \{v \in V | k_i v = q^{\langle \alpha_i, \lambda \rangle} v\}$.

It is proved by Lusztig that the algebra automorphisms $T_i$ above ($T_{i,-1}'$ in \cite{lusztig1993introduction}) are compatible with linear maps $T_i$ of $\Uqghat$-modules $V$ below, in the sense that \[T_i(x.v) = T_i(x).T_i(v), \quad \forall x \in \Uqghat, \; v \in V.\]

Here $T_i : V \to V$ is defined by its action on weight vectors:
    \[T_i(v) = \sum_{a-b+c = \lambda(\alpha_i^{\vee})} (-1)^b q_i^{-(-ac+b)}f_i^{(a)}e_i^{(b)}f_i^{(c)}v , \quad \forall v \in V_{\lambda}, \forall i \in \hat{I}.\]
These linear maps $T_i : V \to V$ are called symmetries of $\Uqghat$-modules. We use the same notation $T_i$ here without risk, since the linear maps $T_i$ act on modules, while Lusztig's automorphisms act on the algebra.

If we replace $T_i$ by another Lusztig's automorphism $T_{i,1}'' = (T_{i,-1}')^{-1} = T_i^{-1}$, then it is compatible with the linear map $T_i^{-1} : V \to V$ given by
    \[T_i^{-1}(v) = \sum_{-a+b-c = \lambda(\alpha_i^{\vee})} (-1)^b q_i^{(-ac+b)}e_i^{(a)}f_i^{(b)}e_i^{(c)}v , \; \forall v \in V_{\lambda},\]
which is the inverse map of $T_i$.

We consider the symmetries acting on extremal vectors. Let $V$ be a finite-dimensional representation of $\Uqghat$. Let $v \in V$ be a weight vector of weight $\lambda$. We
call $v$ an extremal vector if the weights of $V$ are contained in the convex hull of
$W\lambda$ \cite{akasaka1997finite}.

    \begin{lemma}\cite[Proposition~5.2.2]{lusztig1993introduction}
        For an $i \in I$, if $v \in V$ is such that $e_i.v = 0$ and $k_i.v = q_i^m v$, then 
            \[T_i(f_i^{(j)}.v) = (-1)^jq_i^{-j(m+1-j)} f_i^{(m-j)}.v. \]
    \end{lemma}
    
In particular, if $v$ is an extremal vector in the weight space $V_{\lambda}$, then $T_i.v$ is also an extremal vector in the weight space $V_{s_i \lambda}$.

Recall also that when $V$ is irreducible $l$-highest weight representation, the extremal weight spaces have dimension one. Therefore, any extremal vector can be obtained from a highest weight vector $v_0$ by the action of a series of linear maps $T_i$, $i \in I$. Let $w \in W$ an element in the finite Weyl group, choose a reduced expression $w = s_{i_1} \cdots s_{i_r}$, then  $T_w v_0 = T_{i_1} \cdots T_{i_r} v_0$ is an extremal vector in the $1$-dimensional weight space $V_{w\lambda_0}$, where $\lambda_0$ is the weight of $v_0$.

%--------------------------subsection--------------------------%
\subsubsection{Relation with Chari's braid group action}
Only in this section, we denote the braid group actions in Definition~\ref{def:braid action on ringY} by $T_i^{Cha}$ to emphasize the difference from Lusztig's automorphisms.

Let $V = L(\bpsi)$ be a finite-dimensional irreducible representation of highest $l$-weight $\bpsi$. Then the $I$-tuple of rational functions $\bpsi$ is a monomial in $Y_{i,a}$. In this case, we are using Chari's definition by the formula \eqref{eq:B action on Y}.

    \begin{proposition}\label{prop: braid Lusztig Chari}
        Let $V = L(\bpsi)$ be a finite-dimensional irreducible representation of highest $l$-weight $\bpsi$, and $v_0$ be a highest weight vector. Let $w = s_{i_1} \cdots s_{i_r}$ be a reduced expression of $w \in W$. Let $T_w = T_{i_1} \cdots T_{i_r}$ be the series of Lusztig's automorphisms and $T_w^{Cha} = T_{i_1}^{Cha} \cdots T_{i_r}^{Cha}$ be the series of Chari's braid group action. Then
        \begin{equation}\label{eq: braid Lusztig Chari}
            T_{w^{-1}}(\phi_j^{\pm}(u)).v_0 = T_{w}^{-1}(\phi_j^{\pm}(u)).v_0 = T_w^{Cha}(\bpsi)v_0.
        \end{equation}
    \end{proposition}
    \begin{proof}
        Let $\lambda_0$ be the highest weight of $V$. Consider the linear map $T_w :V \to V$. We have seen that $T_wv_0 \in V_{w\lambda_0}$ is a vector in the one-dimensional extremal weight space $V_{w\lambda_0}$.

        By Theorem~\ref{thm:Chari l weight}, the extremal vector $T_wv_0$ is an eigenvector of $\phi_i^{\pm}(u)$, and has eigenvalue $\phi_i^{\pm}(u).T_wv_0 = T_w^{Cha}(\bpsi) T_wv_0$. Therefore, 
        \[T_{w}^{-1}(\phi_j^{\pm}(u)).v_0 = T_w^{-1}(\phi_j^{\pm}(u).T_w(v_0)) = T_w^{-1}(T_w^{Cha}(\bpsi)T_w(v_0)) =T_w^{Cha}(\bpsi)v_0.\]
        In conclusion, $v_0$ is an eigenvector of $T_w^{-1}(\phi_i^{\pm}(u))$ whose eigenvalue coincides with the $l$-weight of the space $V_{w\lambda_0}$.

        Moreover, the above argument also holds for any one of the four Lusztig's automorphisms. We repeat the above argument using $T_{i,1}'' = (T_{i,-1}')^{-1}$, then 
        \[(T_{w,1}'')^{-1}(\phi_j^{\pm}(u)).v_0 = (T_{i_1,1}'' \cdots T_{i_r,1}'')^{-1}(\phi_j^{\pm}(u)).v_0 = T_w^{Cha}(\bpsi)v_0.\]
        Therefore, 
        \[T_{w^{-1}}(\phi_j^{\pm}(u)).v_0 = T_{i_r,-1}' \cdots T_{i_1,-1}'(\phi_j^{\pm}(u)).v_0 =(T_{i_1,1}'' \cdots T_{i_r,1}'')^{-1}(\phi_j^{\pm}(u)).v_0 = T_w^{Cha}(\bpsi)v_0.\]
    \end{proof}

    \begin{remark}
        We remark that recently a relation between Lusztig's automorphisms and Chari's braid group actions is discovered by Friesen, Weekes and Wendlandt \cite{friesen2024braid}. More precisely, even if the Lusztig's automorphisms do not preserve the affine Cartan subalgebra $\Uqhaffine$, the image composed with the projection along the triangular decomposition is a well-defined braid group action on $\Uqhaffine$, and this action is given by Chari's braid group action. This relation may be useful for studying the convergence of projected limits here.
    \end{remark}

%--------------------------subsection----------------------------%

\subsection{Inductive systems}\label{subsec:inductive system}
In this section, we construct inductive systems of representations $T_w^*L(M_km)$ pulled back by a Lusztig's automorphism. The inductive limit is equipped with a well-defined $\Uqb$-action.

\subsubsection{Recall of the construction of Hernandez-Jimbo}
Fix an index $i \in I$ and let $M_k$ be the KR $l$-weight as in \eqref{eqn:KR weight}. Let $m$ be a fixed monomial in $Y_{j,b}$ ($j \in I, b \in \Ci^*$), meaning that $m$ is independent of $k \in \N$. Let $V_k$ be the underlying vector space of $L(M_km)$. 

When $k >l$, $M_kM_l^{-1} = Y_{i,aq_i^{-2k+1}} \cdots Y_{i,aq_i^{-2l-1}}$ is a monomial in variables  $Y_{j,b}$ ($j \in I, b \in \Ci^*$). It is known that there exists an integer $n_0 \in \N$ determined by the monomial $m$, such that for any $k > l > n_0$, there is morphism of $\Uqghat$-modules
	\[L(M_lm) \otimes L(M_kM_l^{-1}) \to L(M_kM_l^{-1}) \otimes L(M_lm)\]
whose image is isomorphic to $L(M_km)$ \cite[Corollary~5.5]{hernandez2010simple}.

Moreover, it is proved that its restriction to 
	\[L(M_lm) \otimes v' \to L(M_kM_l^{-1}) \otimes L(M_lm)\]
is injective, where $v'$ is a highest weight vector of $L(M_kM_l^{-1})$ \cite[Proposition~5.6]{hernandez2010simple}. This defines a family of injective linear maps 
	\begin{equation}
	\varphi_{l,k}: V_l \to V_k, \quad \forall k > l > n_0.
	\end{equation}

We begin by recalling the construction of inductive systems by Hernandez and Jimbo. The following result was proved in \cite{hernandez2012asymptotic}. For the formulation, see also the proof in \cite[Section~3.3]{wang2023qq}.

Let $l_1 < k$ and $l_2 < k$ be positive integers. For $x \in \Uqghat$, we consider the linear map $\varphi_{l_1,k}^{-1}x\varphi_{l_2,k} \in \Hom(V_{l_2},V_{l_1})$ whenever $x(\mathrm{Im}(\varphi_{l_2,k})) \subset \mathrm{Im}(\varphi_{l_1,k})$.
    \begin{lemma}\cite{hernandez2012asymptotic}\label{lemma:operators on limit}
        For all $j \in I$, $\varphi_{l_1,k}^{-1}e_j\varphi_{l_2,k}$, $\varphi_{l_1,k}^{-1}e_0\varphi_{l_2,k}$, and $\varphi_{l_1,k}^{-1}k_j^{-1}f_j\varphi_{l_2,k}$ are linear maps from $V_{l_2}$ to $V_{l_1}$ having the form $C_{l_1,l_2} + q_i^{-2k}D_{l_1,l_2} + \cdots + q_i^{-2kN}E_{l_1,l_2}$ for some $N \geq 1$. Here $C_{l_1,l_2}, D_{l_1,l_2}, \cdots, E_{l_1,l_2}$ stand for linear maps in $\Hom(V_{l_2},V_{l_1})$ depending only on $l_1,l_2$ and the element $x = e_j,e_0$ or $k_j^{-1}f_j$, but are independent of $k$.
    \end{lemma}

As a consequence, it is proved that the corresponding operators $C_{l_1,l_2}$ are well-defined on the inductive limit $\varinjlim V_k$. We denote these operators on $\varinjlim V_k$ correspondingly by $\tilde{e}_j$, $\tilde{e}_0$ and $\tilde{f}_j$, $j \in I$. 

For any algebraic relation of $e_j$, $e_0$ and $k_j^{-1}f_j$ in $\Uqghat$, the corresponding operators on $\varinjlim V_k$ satisfy the same relation \cite[Section~4.2]{hernandez2012asymptotic}. In particular, the linear operators $\tilde{e}_j$ ($j \in I$) and $\tilde{e}_0$ on $\varinjlim V_k$ satisfy the quantum Serre relations \eqref{eq:quantumserrerelations}. 

Moreover, an action of the Cartan subalgebra $\Uqh$ on $\varinjlim V_k$ is defined in the following way: the action of $k_j$ ($\forall j \neq i$) and the action of $q_i^{-k}k_i$ on each $L(M_km)$ commute with the linear maps $\varphi_{l,k}$. Consequently, these operators are well-defined on the inductive limit $\varinjlim V_k$, denoted by $\tilde{k}_j$ ($j \neq i$) and $\tilde{k}_i$ respectively. Moreover, the operators $\tilde{k}_j$ ($j \in I$) satisfy the quantum Weyl relation with the operators $\tilde{e}_j$ ($j \in I$) and $\tilde{e}_0$ defined above.

Recall that by Proposition~\ref{prop:Borel def}, these are defining relations for the quantum affine Borel algebra $\Uqb$. Therefore, the operators $\tilde{e}_j$, $\tilde{k}_j$ ($j \in I$) and $\tilde{e}_0$ define a $\Uqb$-action on $\varinjlim V_k$.

Recall also that $\varphi_{l,k}$ maps a highest weight vector in $L(M_lm)$ to a highest weight vector in $L(M_km)$. This gives $v_0 \in \varinjlim V_k$ a vector of highest weight in the inductive limit. Then the eigenvalue of $\phi_j^+(u)$-action on $v_0$ is $\bpsi_{i,a}^{-1}m$, where 
\[
\bpsi_{i,a}^{\pm 1} = (1,\cdots,(1-au)^{\pm 1},\cdots,1)
\]
defined in \eqref{eq:psi pm}. This eigenvalue can be calculated as the limit of rational functions $M_km$ when $k \to \infty$ if $|q| > 1$ \cite[Remark~4.3]{hernandez2012asymptotic}.

Details of the above argument can also be found in \cite[Section~3.3]{wang2023qq}.

\subsubsection{New inductive systems}
Now we can generalize this construction twisted by $w \in W$. The length of $w \in W$ is denoted by $l(w)$.
    \begin{lemma}[{\cite[Lemma~40.1.2]{lusztig1993introduction}}]\label{lemma:Lusztig U+}
        Let $T_w$ be the Lusztig's automorphism associated with $w \in W$.  Let $i \in I$, if $l(ws_{i}) = l(w) + 1$, then
        \begin{enumerate}
            \item $T_w(e_i)$ is an algebraic combination of $e_j$, $j \in I$.
            \item $T_w(f_i)$ is an algebraic combination of $f_j$, $j \in I$.
        \end{enumerate}
    \end{lemma}
    This result is proved by any quantum Kac-Moody algebra. In particular, it also holds when we replace $W$ by the affine Weyl group associated to $\hatg$ and replace the index set $I$ by $\hat{I}$.

    \begin{proposition}\label{prop:lusztig converge}
        Let $T_w$ be the Lusztig's automorphism associated with $w \in W$. Then
        \begin{enumerate}
            \item $T_w(e_i)$ is an algebraic combination of $e_j$, $j \in I$, if $l(ws_{i}) = l(w) + 1$.
            \item $T_w(e_0)$ is an algebraic combination of $e_j$, $j \in \hat{I}$.
            \item $T_w(e_i)$ is an algebraic combination of $k_j^{-1}f_j$, $j \in I$, if $l(ws_{i}) = l(w) - 1$.
        \end{enumerate}
    \end{proposition}
    \begin{proof}
        The first case is the first point in Lemma~\ref{lemma:Lusztig U+}.
        
        The second case is obtained by applying Lemma~\ref{lemma:Lusztig U+} to $\Uqghat$. Notice that for any element $w$ whose reduced expressions involve only indices in $I$, we have $l(ws_0) = l(w) + 1$.
        
        For the third case, if $l(ws_{i}) = l(w) - 1$, then $w$ has a reduced expression ending with $s_i$: $w = w's_i$ such that $l(w) = l(w') + 1$. Then $T_w(e_i) = T_{w'}(k_i^{-1}f_i) = k_{w'(\alpha_i)}^{-1} T_{w'}(f_i)$. By Lemma~\ref{lemma:Lusztig U+}, since $l(w's_i) = l(w')+1$, $T_{w'}(f_i)$ is an algebraic combination of $f_j$, $j \in I$. Moreover, by definition of Lusztig's automorphisms, each monomial in the algebraic combination of $T_{w'}(f_i)$ has weight $w'(-\alpha_i)$. Therefore, $k_{w'(\alpha_i)}^{-1} T_{w'}(f_i)$ is an algebraic combination of $k_j^{-1}f_j$, $j \in I$.
    \end{proof}

    Recall that $V_k$ is the underlying vector space of $L(M_km)$. Let $\lambda_0 \in \mathfrak{h}^*$ be the weight of $m$.
    \begin{theorem}\label{thm:inductive system}
        For each $w \in W$, there is a well-defined $\Uqb$-action on the inductive limit $\varinjlim V_k$ of the inductive system  $(V_k, \varphi_{l,k})$, such that its weights are contained in the cone $w(\lambda_0) + w\Lambda_-$, and its $w(\lambda_0)$-weight space has dimension one, denoted by $\Ci v_0$, and the eigenvalue of $\phi_j^+(u)$-action on $v_0$ is given by $T_w(\bpsi_{i,a}^{-1}m)$.
    \end{theorem}
    \begin{proof}
        Consider the representation $T_{w^{-1}}^*L(M_km)$ of $\Uqghat$, where each element $e_j$ acts by its image under Lusztig's automorphism $T_{w^{-1}}e_j$. By Proposition~\ref{prop:lusztig converge}, $T_{w^{-1}}e_j$ is an algebraic combination of $e_l$ and $k_l^{-1}f_l$. Therefore, by Lemma~\ref{lemma:operators on limit}, $\varphi_{l_1,k}^{-1}T_{w^{-1}}e_j\varphi_{l_2,k}$ is also of the form $C + q_i^{-2k}D + \cdots q_i^{-2kN}E$. 

        For Serre relations between $e_j$, the elements $T_{w^{-1}}e_j$ also verify these Serre relations since $T_{w^{-1}}$ is an algebra automorphism. By the same proof as in \cite[Section~4.3]{hernandez2012asymptotic} and in \cite[Theorem~3.15]{wang2023qq}, the corresponding operators $C$ in $\varphi_{l_1,k}^{-1}T_{w^{-1}}e_j\varphi_{l_2,k}$ verify these Serre relations. 

        We complete the $k_j$-actions on the space $\varinjlim V_k$ in the following way. Let $\lambda_0$ be the weight of the monomial $m$. Then weights of the module $L(M_km)$ are contained in $\lambda_0 + k\omega_i + \Lambda_-$. For each weight vector $v \in V_k = L(M_km)$ of weight $\lambda$, define new actions of Cartan elements $k_j$ on $v$ by $q^{\langle \lambda - k\omega_i, w^{-1}\alpha_j \rangle} v = q^{\langle w(\lambda - k\omega_i), \alpha_j \rangle} v$. Then these $k_j$ actions commute with $\varphi_{l,k}$, and on each $V_k$, the weights of these $k_j$-actions are contained in the cone $w(\lambda_0) + w\Lambda_-$.

        In conclusion, these operators give a $\Uqb$-action on the inductive limit $\varinjlim V_k$. Its weights are contained in the cone $w(\lambda_0) + w\Lambda_-$ and the $w(\lambda_0)$-weight space has dimension one, since each component $V_k$ is so.

        Finally we calculate the $\phi_j^+(u)$ actions on $v_0$. By Proposition~\ref{prop: braid Lusztig Chari}, the $T_{w^{-1}}\phi_{j}^+(u)$ action on $v_0$ is given by $T_w(M_km)$. It means that $\forall j \in I, n \geq 0$, each $\phi_{j,n}$-action is of the form $C + q_i^{-2k}D + \cdots$, where $C,D$ acts on $v_0$ by multiplication. If we assume $|q| > 1$, the operator $C$ can be calculated as the limit of the $I$-tuple rational function $T_w(M_km)$ when $k \to \infty$, which is nothing but $T_w(\bpsi_{i,a}^{-1}m)$.
    \end{proof}

	As a special case, when $m =1$, it is known that the dimension of each weight space of KR modules $L(M_k)$ stabilizes as $k \to \infty$ \cite{kuniba2002canonical}, we have
    \begin{corollary}\label{cor:weights of inductive limit}
        When $m = 1$, the $\Uqb$-module $\varinjlim V_k$ constructed above has a weight space decomposition with finite-dimensional weight spaces, and all weights of $\varinjlim V_k$ are contained in $w\Lambda_-$. Moreover, the zero weight space has dimension $1$, and the eigenvalue of $\phi^+_i(u)$-actions on the zero weight space is given by $T_w(\bpsi_{i,a}^{-1})$.
    \end{corollary}

    \begin{definition}
        We denote the $\Uqb$-module $\varinjlim V_k$ constructed above by $V^w_{\infty}(m)$. When the monomial $m = 1$, we simply denote it by $V^w_{\infty}$.
    \end{definition}

Let us calculate some examples of the $l$-weight $T_w(\bpsi_{i,a}^{-1}m)$.
    \begin{example}
        Let $\mathfrak{g} = \mathfrak{sl}_3$ and the fixed index $i = 1$.

        We take $L(M_km)$ to be the KR module $W_{1,q^{-2k+1}}^{(k)} = L(Y_{1,q^{-2k+1}}\cdots Y_{1,q^{-1}})$. Let $v_0$ be a highest weight vector, then 
        \[T_1(\phi^+(u)).v_0 = (q^{-k}\frac{1-q^2u}{1-q^{-2k+2}u}, q^{k}\frac{1-q^{-2k+1}u}{1-qu})v_0,\]
        \[T_2T_1(\phi^+(u)).v_0 = (1, q^{-k}\frac{1-q^{3}u}{1-q^{-2k+3}u})v_0.\]

        Suppose $q \in \mathbb{C}^*$ so that $|q| > 1$, the above eigenvalues, with the modified $k_j$-actions, converge in $\Ci \llbracket u \rrbracket$ as $k \to \infty$:
            \begin{itemize}
                \item the eigenvalues on $v_0$ in $T_1^*W_{1,q^{-2k+1}}^{(k)}$ tend to $\bpsi_{1,q^2} \bpsi_{2,q}^{-1}$,
                \item the eigenvalues on $v_0$ in $(T_1T_2)^*W_{1,q^{-2k+1}}^{(k)}= T_2^*T_1^*W_{1,q^{-2k+1}}^{(k)}$ tend to $\bpsi_{2,q^3}$.
            \end{itemize}
        
        As a consequence, we have constructed in Theorem~\ref{thm:inductive system} two infinite-dimensional representation of the subalgebra $\Uqb$ of $\mathcal{U}_q \hat{\mathfrak{sl}}_3$:
            \begin{itemize}
                \item a representation $V_{\infty}^{s_1}$ whose weights are contained in $s_1\Lambda_-$, and whose $l$-weight of the $1$-dimensional zero weight space is given by $\bpsi_{1,q^2} \bpsi_{2,q}^{-1}$,
                \item a representation $V_{\infty}^{s_2s_1}$ whose weights are contained in $s_2s_1\Lambda_-$, and whose $l$-weight of the $1$-dimensional zero weight space is given by $\bpsi_{2,q^3}$.
            \end{itemize}
    \end{example}

    We calculate another example which will be used in Section~\ref{subsec:relation Ow and O}.
    \begin{example}\label{ex:minimal affinization sl3}
    	Let $\g = \mathfrak{sl}_3$ and the fixed index $i = 1$. In this example, we take $l \in \N$ to be a fixed integer and the monomial $m = Y_{2,q^2}Y_{2,q^4}\cdots Y_{2,q^{2l}}$ independent of $k$, then 
        \[L(M_km) = L(Y_{1,q^{-2k+1}}\cdots Y_{1,q^{-1}}Y_{2,q^2}\cdots Y_{2,q^{2l}}).\] 
        
        In this example,
        \[T_1(\phi^+(u)).v_0 = (q^{-k}\frac{1-q^2u}{1-q^{-2k+2}u}, q^{k+l}\frac{1-q^{-2k+1}u}{1-q^{2l+1}u})v_0,\]
        \[T_2T_1(\phi^+(u)).v_0 = (q^l\frac{1-q^2u}{1-q^{2l+2}u}, q^{-k-l}\frac{1-q^{2l+3}u}{1-q^{-2k+3}u})v_0,\]
        \[T_1T_2T_1(\phi^+(u)).v_0 = (q^{-l}\frac{1-q^{2l+4}u}{1-q^{4}u}, q^{-k}\frac{1-q^{3}u}{1-q^{-2k+3}u})v_0.\]

        Therefore,
        \begin{itemize}
            \item when $w = s_1$, the $l$-weight of $v_0$ in the inductive limit is $e^{l\omega_2}\bpsi_{1,q^2} \bpsi_{2,q^{2l+1}}^{-1}$. 
            \item when $w = s_2s_1$, the $l$-weight of $v_0$ in the inductive limit is $e^{l(\omega_1 - \omega_2)}\bpsi_{1,q^2} \bpsi_{1,q^{2l+2}}^{-1} \bpsi_{2,q^{2l+3}}$. 
            \item when $w = s_1s_2s_1$, the $l$-weight of $v_0$ in the inductive limit is $e^{-l\omega_1}\bpsi_{1,q^{2l+4}} \bpsi_{1,q^4}^{-1} \bpsi_{2,q^3}$.
        \end{itemize}
    \end{example}

%----------------------subsection----------------------------------%
\subsection{Projective systems}
Recall that in \cite[Section~7.4]{hernandez2012asymptotic}, $\Uqb$-actions are also defined on projective limits of projective systems $(L(N_k),\Pi_{k,l})$. Here 
\[N_k = Y_{i,q_i}Y_{i,q_i^3} \cdots Y_{i,q_i^{2k-1}},\] 
and $\Pi_{k,l} : L(N_k) \to L(N_l)$ is surjective, $\forall k \geq l$.

In \cite{hernandez2012asymptotic}, it is proved that in the projective system, $\Pi_{k,l'}^{-1}e_jk_j^{-1}\Pi_{k,l}$ and $\Pi_{k,l'}^{-1}f_jk_j^{2}\Pi_{k,l}$ are linear operators of the form $C + q_i^{2k}D + \cdots + q_i^{2kN}E$ for some $N \geq 1$. Again $C,D,\cdots,E$ stand for operators independent of $k$. As a consequence, the projective limit $\varprojlim L(N_k)$ is equipped with a $\Uqb$-action which is an $l$-highest weight module of highest weight $\bpsi_{i,a}$.

Just as with inductive systems, we can generalize this construction with a Weyl group twist, except that we use Lusztig's automorphism $T_i''$ instead of $T_i$.

We write the results without repeating the proof.

\begin{theorem}\label{thm:projective systems}
	Lemma~\ref{lemma:Lusztig U+} also holds for  $T_i''$, and by repeating the proof, we have
    \begin{enumerate}
        \item Analogous to Proposition~\ref{prop:lusztig converge}, $\forall w \in W$, the element $T_w''(e_jk_j^{-1})$, for all $j \in \hat{I}$, is either an algebraic combination of $e_jk_j^{-1}$, or an algebraic combination of $f_jk_j^{2}$.
        \item Analogous to Theorem~\ref{thm:inductive system}, let $W_k$ be the underlying space of $L(N_k)$, then the projective limit $\varprojlim W_k$ is equipped with a $\Uqb$-action.
        \item Moreover, weights of this module $\varprojlim W_k$ are contained in $w\Lambda_-$, the zero weight space has dimension $1$, and the eigenvalue of $\phi_i^+(u)$-action on this $1$-dimensional space is given by $T_w(\bpsi_{i,a})$.
    \end{enumerate}
\end{theorem}
%--------------------------------------------------------
%section
%--------------------------------------------------------

\section{Categories \texorpdfstring{$\catO^w$}{} of representations of Borel subalgebras}\label{sec:Borel subalgebras}
We define categories $\catO^w$ of representations of Borel subalgebras so that the inductive limits $V^w_{\infty}(m)$ constructed in the previous section are $\Uqb$-modules in these categories $\catO^w$. We establish a classification of simple objects in these categories (Theorem~\ref{thm:classification catOw}), and we propose a conjecture which connect the $q$-characters of representations in $\catO^w$ and the projected limits in Section~\ref{sec:limits of q characters}. 

%--------------------------subsection-----------------------%
\subsection{Category \texorpdfstring{$\catO^w$}{}}
Recall that the $\Uqb$-modules in the category $\catO$ are modules with weight space decomposition such that the weights are contained in a finite union $\bigcup_{\lambda} (\lambda +\Lambda_-)$. The simple objects in $\catO$ are classified by $l$-highest weight modules.

Motivated by \cite{frenkel2022weyl}, for any Weyl group element $w$, we define categories $\catO^w$ to be the category of $\Uqb$-modules just as $\catO$ while the weights contained in a finite union $\bigcup_{\lambda} (\lambda + w\Lambda_-)$.

    \begin{definition}\label{def:category Ow}
        Let $\Uqb$-$\mathrm{Mod}$ be the category of $\Uqb$-modules. For $w \in W$, the categories $\catO^w$ is the full subcategory of $\Uqb$-$\mathrm{Mod}$ whose objects are $\Uqb$-modules such that
        \begin{itemize}
            \item they have a weight space decomposition $V = \bigoplus_{\lambda \in \mathfrak{h}^*} V_{\lambda}$, where $V_{\lambda} := \{v \in V | k_i v = q^{\langle \alpha_i, \lambda \rangle} v\}$,
            \item with each weight space $V_{\lambda}$ of finite dimension,
            \item and all weights $\{\lambda \in \mathfrak{h}^* | V_{\lambda} \neq 0 \}$ are contained in a finite union $\cup_i (\lambda_i + w\Lambda_-)$, where $\lambda_i \in \mathfrak{h}^*$.
        \end{itemize}
    \end{definition}

In particular, when $w = e$ the neutral element, the category $\catO^{e}$ is the category $\catO$ of Hernandez-Jimbo \cite[Section~3.3]{hernandez2012asymptotic}.

By the inductive limits constructed in Theorem~\ref{thm:inductive system}, the categories $\catO^w$ are non-empty.

\begin{lemma}\label{lemma:tensor Ow}
	For each $w \in W$, the category $\catO^w$ is a tensor category which is closed under sub objects and quotients.
\end{lemma}
\begin{proof}
	It follows directly from the definition that submodules, quotient modules and tensor products of $\Uqb$-modules in $\catO^w$ also have finite-dimensional weight space decompositions, whose weights are contained in a finite union $\cup_i (\lambda_i + w\Lambda_-)$. 
\end{proof}

We define the notion of $w$-highest weight modules which will be used to classify simples in the category $\catO^{w}$.

Let $\Phi$ (resp. $\Phi_+$) be the set of roots (resp. positive roots) of $\g$, and $\hat{\Phi}$ (resp. $\hat{\Phi}_+$) the set of roots (resp. positive roots) of $\hatg$.

Recall that for each real affine root $\alpha + m\delta \in \hat{\Phi}$, where $\alpha \in \Phi$ and $m \in \Z$, a root vector $E_{\alpha + m\delta} \in \Uqghat$ is defined via Lusztig's automorphisms $T''_i$ \cite{beck1994braid,beck1994convex}. Moreover, the real root vector $E_{\alpha + m\delta}$ is contained in the Borel subalgebra $\Uqb$ if and only if $\alpha + m\delta \in \hat{\Phi}_+$.

    \begin{definition}
        A $\Uqb$-module $V$ is called a $w$-highest weight representation if it is generated by a vector $v_{w} \in V$ such that 
        \begin{enumerate}
            \item $E_{\alpha + m\delta}.v_{w} = 0$ for all root vectors such that $\alpha + m\delta \in \hat{\Phi}_+ \cap (w\Phi_+ + \Z \delta)$,
            \item $\phi_{i,m}^+.v_w = \psi_{i,m}v_w$, $\forall i \in I, m \geq 0$, for some $\psi_{i,m} \in \Ci$.
        \end{enumerate}
         In this case, we call $v_w$ a $w$-highest weight vector. We call $\bpsi = (\psi_i(u))_{i \in I}$ a $w$-highest $l$-weight of $V$.
    \end{definition}

In particular, when $w = e$, an $e$-highest weight representation is nothing but what we called an $l$-highest weight representation.

	\begin{lemma}\label{lemma:w weight space dim one}
		If $V$ is a $w$-highest weight representation of $w$-highest $l$-weight $\bpsi$, then it has a unique $w$-highest weight vector $v_w$ up to a constant, and the weight space of $v_w$ has dimension one.
	\end{lemma}
	\begin{proof}
		 Recall that we have a PBW basis of $\Uqb$ formed by root vectors \cite{lusztig1990finite,beck1994braid}: 
		\begin{itemize}
			\item $E_{-w(\alpha) + n\delta}$ ($\alpha \in \Phi_+$, $n \geq 0$ when $-w(\alpha) \in \Phi_+$, and $n > 0$ when $-w(\alpha) \in -\Phi_+$),
			\item $E_{(m\delta,i)}$ ($m \geq 0$, $i \in I$),
			\item $E_{w(\alpha) + n\delta}$ ($\alpha \in \Phi_+$, $n \geq 0$ when $w(\alpha) \in \Phi_+$, and $n > 0$ when $w(\alpha) \in -\Phi_+$),
		\end{itemize}
		where root vectors in the third line coincide with $E_{\alpha + m\delta}$ such that $\alpha + m\delta \in \hat{\Phi}_+ \cap (w\Phi_+ + \Z \delta)$. 
		
		Moreover, the imaginary root vectors in the second line coincide with  $\phi^+_{i,m}$, $i \in I, m \in \N$ \cite{beck1994braid}. 
		
		Let $\lambda_0$ be the weight of $\bpsi$. Let $v_w$ be a $w$-highest weight vector of $V$, then $v_w \in V_{\lambda_0}$. Using the PBW basis, one see that weights of the $\Uqb$-module $V$ are contained in $\lambda_0 + w\Lambda_-$ and the $\lambda_0$ weight space $V_{\lambda_0}$ is $1$-dimensional. Thus the $w$-highest weight vector $v_w$ is unique up to a constant.
	\end{proof}

    \begin{proposition}\label{prop:unique LwM}
        If there exists a $w$-highest weight representation $V$ in the category $\catO^w$ of $w$-highest $l$-weight $\bpsi$, then there is a unique irreducible $w$-highest weight representation in the category $\catO^w$ of the $w$-highest $l$-weight $\bpsi$, which is a quotient of $V$.
    \end{proposition}
    \begin{proof}
    	For any $w$-highest weight representation $M$ in the category $\catO^w$, let $v_w \in M$ be a $w$-highest weight vector. By definition, $M$ is generated from $v_w$ as $\Uqb$-module. Thus there is a surjective morphism of $\Uqb$-modules $\Uqb \to M$ which maps $1$ to $v_w$. Let $I_M \subset \Uqb$ be the kernel of this morphism. 
        
        Consider the set $\mathcal{P}$ of all $w$-highest weight representations in the category $\catO^w$ of the given $w$-highest $l$-weight $\bpsi$. The set $\mathcal{P}$ is non-empty by assumption. 
        
        Let $I_{\bpsi} = \bigcap_{M \in \mathcal{P}} I_M$ be the intersection of kernels defined above. By definition, $I_{\bpsi}$ contains real root vectors $E_{\alpha + m\delta}$ such that $\alpha + m\delta \in \hat{\Phi}_+ \cap (w\Phi_+ + \Z \delta)$. Moreover, $I_{\bpsi}$ contains elements $\phi_{i,m}^+ - \psi_{i,m}$, $\forall i \in I, m \in \N$.
        
        Define $M(\bpsi) := \Uqb / I_{\bpsi}$, then $M(\bpsi)$ is a $w$-highest weight representation, and the image of $1$ in the quotient $\Uqb / I_{\bpsi}$ is a $w$-highest weight vector. Let $\lambda_0$ be the weight of $\bpsi$. By Lemma~\ref{lemma:w weight space dim one}, weights of the $\Uqb$-module $M(\bpsi)$ are contained in $\lambda_0 + w\Lambda_-$, and the weight space $M(\bpsi)_{\lambda_0}$ is $1$-dimensional. 
        
        Since the image of $1$ in $\Uqb / I_{\bpsi}$ is a $w$-highest weight vector of $l$-weight $\bpsi$, it is a vector in the $1$-dimensional space $M(\bpsi)_{\lambda_0}$. Thus any non-zero vector $v \in M(\bpsi)_{\lambda_0}$ generates $M(\bpsi)$. Therefore, any proper submodule of $M(\bpsi)$ can not contain the $\lambda_0$ weight space. Thus $M(\bpsi)$ admits a maximal proper submodule, which is the sum of all proper submodules of $M(\bpsi)$, denoted by $S$. 
        
        Consider the $\Uqb$-module $M(\bpsi)/S$. It follows from definition that $M(\bpsi)/S$ is a $w$-highest weight representation of $w$-highest $l$-weight $\bpsi$, and $M(\bpsi)/S$ is irreducible. Let $V$ be a $w$-highest weight representation in $\catO^w$ in the assumption, then $V \simeq \Uqb /I_V$ is a quotient of $M(\bpsi)$, thus the kernel of $M(\bpsi) \to V$ is a proper submodule of $M(\bpsi)$, which is then contained in the maximal submodule $S$. This gives a non zero morphism $V \to M(\bpsi)/S$, which must be surjective as $M(\bpsi)/S$ is irreducible. Since we have assumed $V$ to be a representation in $\catO^w$, $M(\bpsi)/S$ as a quotient of $V$ also lies in $\catO^w$.
        
        Moreover, any $w$-highest weight representation of $w$-highest $l$-weight $\bpsi$ admits $M(\bpsi)/S$ as a quotient. Thus $M(\bpsi)/S$ is the unique irreducible $w$-highest weight representation of $w$-highest $l$-weight $\bpsi$.
    \end{proof}

As a consequence of the above proposition, the following definition is well-defined.
    \begin{definition}\label{def:LwM}
        Whenever it exists, we denote $L_{w}(\bpsi)$ to be the unique irreducible $w$-highest weight representation in $\catO^w$ with $w$-highest $l$-weight $\bpsi$.
    \end{definition}

%----------------------subsection------------------------%
\subsection{Classification of simple modules in \texorpdfstring{$\catO^w$}{}}
In this section, we classify simple objects in the categories $\catO^w$.

    \begin{lemma}\label{lemma:w negative prefundamental}
        There exist in the category $\catO^w$ the irreducible $w$-highest weight representations $L_{w}(T_w(\bpsi_{i,a}^{-1}))$ and $L_{w}(T_w(\bpsi_{i,a}))$.
    \end{lemma}
    \begin{proof}
        For each fixed index $i \in I$, by Corollary~\ref{cor:weights of inductive limit}, the inductive limit $V_{\infty}^w$ constructed in Theorem~\ref{thm:inductive system} is a $\Uqb$-module in the category $\catO^w$.

        The weights of $V_{\infty}^w$ are contained in $w\Lambda_-$ and the zero weight space is $\Ci v_0$. The $l$-weight of $v_0$ is $T_w(\bpsi_{i,a}^{-1})$. Therefore, the irreducible $w$-highest weight module $L_{w}(T_w(\bpsi_{i,a}^{-1}))$ is a quotient of the submodule of $V_{\infty}^w$ generated by $v_0$. Thus $L_{w}(T_w(\bpsi_{i,a}^{-1}))$ lies in the category $\catO^w$.

        Similarly, $L_{w}(T_w(\bpsi_{i,a}))$ is a subquotient of the projective limit constructed in Theorem~\ref{thm:projective systems}. Thus $L_{w}(T_w(\bpsi_{i,a}))$ also lies in the category $\catO^w$.
    \end{proof}

    \begin{theorem}\label{thm:simples are whw}
        Any irreducible representation in $\catO^w$ is a $w$-highest weight representation.
    \end{theorem}
    \begin{proof}
        Let $V$ be an irreducible representation in $\catO^w$. We say a weight $\mu$ of $V$ is maximal if all elements in $\mu + w\Lambda_+$ other than $\mu$ are not weights of $V$. 
        
        Since $\g$ is of finite type, there can appear only finitely many maximal weights in a cone $\lambda_i + w\Lambda_-$. By definition, all weights of $V$ are contained in a finite union $\cup_i (\lambda_i + w\Lambda_-)$, thus there are finitely many maximal weights. Let $V'$ be the direct sum of these maximal weight spaces. Then $V'$ is finite-dimensional since each weight space is finite-dimensional.

        $V'$ is non zero since weights of $V$ are contained in a finite union $\cup_i (\lambda_i + w\Lambda_-)$. Moreover, any vector $v \in V'$ is killed by root vectors $E_{\alpha + m\delta}$ such that $\alpha + m\delta \in \hat{\Phi}_+ \cap (w\Phi_+ + \Z \delta)$ because $E_{\alpha + m\delta}.v \in V_{\mathrm{wt}(v)+\alpha} = 0$. Furthermore, $\phi_{i,n}^+$ preserves the space $V'$ for weight reason.

        Since the $\phi_{i,n}^+$ commute, they have a common eigenvector $v_w$ in the finite-dimensional space $V'$. Since $V$ is irreducible, the submodule generated by $v_w$ is $V$ itself. Thus $V$ is a $w$-highest weight representation.
    \end{proof}

    \begin{proposition}[{\cite[Lemma~3.9]{hernandez2012asymptotic}}]\label{prop:rational functions}
        If $V$ is a representation in the category $\catO^w$, if $v \in V$ is a common eigenvector of $\phi_j^+(u)$, whose eigenvalue is given by an $I$-tuple of formal power series $\bpsi = (\psi_i(u))_{i \in I}$. Then each $\psi_i(u) \in \Ci \llbracket u \rrbracket$ is a rational function.
    \end{proposition}

Remark that the constant terms $\psi_{i,0}$ of $\psi_i(u)$ give the eigenvalue of $k_i$, thus are non-zero. Therefore, $\bpsi$ is an $I$-tuple of rational functions which are regular and non-zero at $0$.

	\begin{lemma}\label{lemma:tensor of weights}
		Let $L_w(\bpsi)$ and $L_w(\bpsi')$ be two irreducible $w$-highest weight representations in $\catO^w$. Let $v$ and $v'$ be $w$-highest weight vectors in $L_w(\bpsi)$ and $L_w(\bpsi')$ respectively. Then $v \otimes v'$ has the $l$-weight $\bpsi \bpsi'$.
	\end{lemma}
	\begin{proof}
		The proof is similar to that of tensor products of $l$-highest weight modules. It is proved \cite{damiani1998r} that $\forall i \in I, k \geq 0$, the coproduct $\Delta(\phi^+_{i,k})$ has the form
		\[\Delta(\phi^+_{i,k}) = \sum_{j=0}^k \phi^+_{i,k-j} \otimes \phi^+_{i,j} + \cdots,\]
		where $\cdots$ is a sum of elements of the form $a \otimes b$ such that $a$ and $b$ are homogeneous elements in $\Uqb$ of degree $\deg(a) \in \Phi_+ + \Z \delta$ and $\deg(b) \in \Phi_- + \Z \delta$ such that $\deg(a) + \deg(b) \in \Z\delta$. Thus, either $a.v$ or $b.v'$ has a weight outside the cone $w\Lambda_-$. Therefore, at least one of $a.v$ and $b.v'$ vanishes, thus $(a \otimes b).(v \otimes v') = 0$.
		
		In conclusion, $\phi^+_i(u).(v \otimes v') = \phi^+_i(u).v \otimes  \phi^+_i(u).v'$, and thus $v \otimes v'$ has $l$-weight $\bpsi \bpsi'$.
	\end{proof}

    \begin{theorem}\label{thm:classification catOw}
        Simple modules in the category $\catO^w$ are exactly the irreducible $w$-highest weight modules $L_w(\bpsi)$ whose $w$-highest $l$-weight $\bpsi$ is an $I$-tuple of rational functions which are regular and non-zero at $0$.
    \end{theorem}
    \begin{proof}
        Given any $\bpsi$ $I$-tuple of rational functions which are regular and non-zero at $0$. By Definition~\ref{def:barid group action on psi},  $T_w^{-1}(\bpsi)$ is also an $I$-tuple of such rational functions. Recall that $\bpsi_{i,a}^{\pm}$ are denoted to be $I$-tuple of rational functions as in \eqref{eq:psi pm}. Therefore we can decompose 
        \[T_w^{-1}(\bpsi) = q^{\lambda} \prod_{(i,a) \in I \times \Ci^*} \bpsi_{i,a}^{n_{i,a}} \prod_{(j,b) \in I \times \Ci^*} \bpsi_{j,b}^{-m_{j,b}},\] 
        where  $\lambda \in \Lambda \otimes_{\Z} \Ci$ and $n_{i,a},m_{j,b} \in \N^*$ so that $n_{i,a} , m_{j,b} \neq 0$ for only finitely many $(i,a),(j,b) \in I \times \Ci^*$. Then 
        \[\bpsi = q^{w(\lambda)} \prod_{(i,a) \in I \times \Ci^*} T_w(\bpsi_{i,a})^{n_{i,a}} \prod_{(j,b) \in I \times \Ci^*} T_w(\bpsi_{j,b}^{-1})^{m_{j,b}}.\] 

        Consider the tensor product of $w$-highest weight modules 
        \begin{equation}\label{eq:tensor of reps}
        \Ci_{w(\lambda)} \otimes \bigotimes_{(i,a) \in I \times \Ci^*} L_w(T_w(\bpsi_{i,a}))^{\otimes n_{i,a}} \otimes \bigotimes_{(j,b) \in I \times \Ci^*} L_w(T_w(\bpsi_{j,b}^{-1}))^{\otimes m_{j,b}},
        \end{equation}
        where $\Ci_{w(\lambda)}$ is the one-dimensional representation of $\Uqb$ where $k_i$ acts by $q^{\langle w(\lambda) ,\alpha_i \rangle}$.

		By Lemma~\ref{lemma:w negative prefundamental}, each factor in the tensor product in \eqref{eq:tensor of reps} is a representation in $\catO^w$. By Lemma~\ref{lemma:tensor Ow}, the tensor product is also a representation in $\catO^w$. By Lemma~\ref{lemma:w weight space dim one}, the weights of the $\Uqb$-module \eqref{eq:tensor of reps} are contained in $w(\lambda) + w\Lambda_-$, and the weight space of weight $w(\lambda)$ is one-dimensional. 
		
		Let $v_w$ be a vector in this one-dimensional weight space. By Lemma~\ref{lemma:tensor of weights}, $v_w$ is an $l$-weight vector of $l$-weight $\bpsi$. The submodule of this tensor product generated by $v_w$ is a $w$-highest weight module in $\catO^w$. By Proposition~\ref{prop:unique LwM}, it admits $L_w(\bpsi)$ as a quotient and thus $L_w(\bpsi) \in \catO^w$.

        Conversely, by Theorem~\ref{thm:simples are whw}, any simple module in $\catO^w$ is of $w$-highest weight, and by Proposition~\ref{prop:rational functions}, the $w$-highest $l$-weight is such an $I$-tuple of rational functions. This completes the proof.
    \end{proof}

%---------------------------
%section
%---------------------------

\subsection{Conjectures on characters of the category \texorpdfstring{$\catO^w$}{}}\label{sec:conjectures on characters}
The appeal of the category $\catO^w$ lies in its close connection with the category $\catO$. We formulate characters for representations in $\catO^w$, and we state their conjectural relation to characters of representations in $\catO$.

\subsubsection{Usual characters}
We begin with the usual characters.

\begin{definition}
    Let $L_w(\bpsi)$ be an irreducible $w$-highest weight representation in the category $\catO^w$. Let $\lambda_0$ be the weight of $\bpsi$, then all weights of $L_w(\bpsi)$ are contained in $\lambda + w\Lambda_-$.
    
    The usual character of $L_w(\bpsi)$ is defined to be
    \[\chi(V) = \sum_{\lambda \in \mathfrak{h}^*} \dim(V_{\lambda})e^{\lambda} \in e^{\lambda_0} \overline{\mathscr{C}}_w.\]
\end{definition}

\begin{remark}\label{remark:hatO}
    We can also study Weyl group twist of the category $\widehat{\catO}$ of infinite-dimensional representations of quantum affine algebras $\Uqghat$ studied in \cite{hernandez2005representations,mukhin2014affinization}. We define the categories $\widehat{\catO}^w$ of $\Uqghat$-modules exactly the same as Definition~\ref{def:category Ow}. These categories $\widehat{\catO}^w$ are directly related to $\widehat{\catO}$ on level of algebras in the following way: Lusztig's automorphisms $T_w$ are automorphisms of the algebra $\Uqghat$, for any representation $V$ in $\widehat{\catO}$, its pull back $T_{w^{-1}}^*V$ is then a representation in $\widehat{\catO}^w$, and vice versa.

    In particular, we have 
        \[\chi(T_{w^{-1}}^*L(\bpsi)) = w(\chi(L(\bpsi)))\]
    when $\bpsi$ in an $l$-weight as in \cite[Theorem~3.7]{mukhin2014affinization}.
\end{remark}

However, this argument can not be generalized to representations of $\Uqb$. We conjecture that the above equality holds also for representations of $\Uqb$. That is, the usual character of representations in $\catO^w$ can be obtained from representations in $\catO$ simply by the Weyl group action on weights.

    \begin{conjecture}\label{conj:usual char w}
        The usual character of $L_{w}(\bpsi)$ can be calculated by:
        \[\chi(L_{w}(\bpsi)) = w(\chi(L(T_w^{-1}(\bpsi))) ).\]
        Here $w$ acts on the usual character by standard Weyl group action, and $T_w$ acts on $\mathbf{\Psi}$ by the braid group action on $\mathfrak{r}$.
    \end{conjecture}

	\begin{proposition}
		Conjecture~\ref{conj:usual char w} is true when $\bpsi$ is an $I$-tuple of rational functions which are regular at $0$ and at $\infty$, and such that $\psi_i(0) \psi_i(\infty) = 1$, $\forall i \in I$.
	\end{proposition}
	\begin{proof}
		In this case, the $T_w^{-1}(\bpsi)$ is also such an $I$-tuple of rational functions. By \cite[Theorem~3.7]{mukhin2014affinization}, the $\Uqb$-module $L(T_w^{-1}(\bpsi))$ is the restriction of a $\Uqghat$-module to Borel subalgebra. Moreover, by the proof of \cite[Lemma~3.4]{feigin2017finite}, for any simple $\Uqghat$-module $V$, any non-zero vector $v \in V$ generates $V$ as $\Uqb$-module, thus $V$ is also simple as a $\Uqb$-module. This proof also holds for representations in the category $\catO^w$.
		
		We denote by $V = L(T_w^{-1}(\bpsi))$ and consider the $\Uqghat$-module $T_{w^{-1}}^*V$. Then $T_{w^{-1}}^*V$ is a simple $\Uqghat$-module in a category $\widehat{\catO}^w$. Let $v_0$ be a highest weight vector in $V = L(T_w^{-1}(\bpsi))$. By \cite[Theorem~6.5]{friesen2024braid}, $v_0$ is an eigenvector of $T_{w^{-1}}\phi_i^{\pm}(u)$, and the associated eigenvalue is given by $T_w(T_w^{-1}(\bpsi)) = \bpsi$. Therefore, the $w$-highest weight module $L_w(\bpsi)$ is the restriction of $T_{w^{-1}}^*V$ to Borel subalgebra $\Uqb$ which is still simple. 
		
		In particular, the usual character $\chi(L_w(\bpsi)) = \chi(T_{w^{-1}}^*V)$, which equals $w(\chi(V))$ as we have seen in Remark~\ref{remark:hatO}.
	\end{proof}

    Conjecture~\ref{conj:usual char w} implies the irreducibility of inductive limits.
    \begin{corollary}[of Conjecture~\ref{conj:usual char w}]
        The inductive limit $V^w_{\infty}$ of the inductive system in Theorem~\ref{thm:inductive system} is the irreducible $w$-highest weight modules $L_{w}(T_w(\bpsi_{i,a}^{-1}))$.
    \end{corollary}
    \begin{proof}
        When $w=e$, it is known that the limit of the inductive system formed by $L(M_k)$ is irreducible and equals to the $l$-highest weight module $L(\bpsi_{i,a}^{-1})$ \cite[Theorem~6.1]{hernandez2012asymptotic}. In particular, the usual character of the limit equals to  
        \[\chi(V_{\infty}^e) = \chi(L(\bpsi_{i,a}^{-1})) = \lim_{k \to \infty} e^{-k\omega_i}\chi(L(M_k)).\]

        For the inductive system formed by $T_{w^{-1}}^*L(M_k)$, on one hand, the usual character of the inductive limit $\chi(V_{\infty}^w)$ equals to
        \[\chi(V_{\infty}^w) = \lim_{k \to \infty}e^{-kw(\omega_i)}\chi(T_{w^{-1}}^*L(M_k)) = \lim_{k \to \infty} w(e^{-k\omega_i}\chi(L(M_k))) = w(\chi(L(\bpsi_{i,a}^{-1}))).\]
        
        On the other hand, if Conjecture~\ref{conj:usual char w} holds, the right hand side equals to $\chi(L_{w}(T_w(\bpsi_{i,a}^{-1})))$.

        In conclusion, the inductive limit $V_{\infty}^w$ admits $L_{w}(T_w(\bpsi_{i,a}^{-1}))$ as a subquotient, while they have the same usual character. Therefore they coincide and the inductive limit is irreducible.
    \end{proof}

\subsubsection{\texorpdfstring{$q$}{}-characters}

We define the $q$-characters for the categories $\catO^w$ as for the category $\catO$ only by changing $\Lambda_-$ to $w\Lambda_-$ \cite{hernandez2012asymptotic}. Let us recall the definition.

\begin{definition}
    Let $\mathfrak{r}$ be the set of $I$-tuples of rational functions which are regular and non-zero at $0$. For $\bpsi = (\psi_i (u))_{i \in I} \in \mathfrak{r}$, denote by $\mathrm{wt}(\bpsi) \in \mathfrak{h}^*$ the weight of $\bpsi$ as in Section~\ref{subsec:recall q char}.
    Let $\mathcal{E}_l^w$ be the set of maps $c : \mathfrak{r} \to \Z$ satisfying:
    \begin{itemize}
    \item $c(\bpsi) = 0$ for all $\bpsi$ such that $\mathrm{wt}(\bpsi)$ is outside a finite union $\cup (\lambda_i + w\Lambda_-)$.
    \item For each $\lambda \in \mathfrak{h}^*$, there are finitely many $\bpsi$ such that $\mathrm{wt}(\bpsi) = \lambda$ and $c(\bpsi) \neq 0$.
    \end{itemize}

    For $\bpsi \in \mathfrak{r}$, let $[\bpsi] \in \mathcal{E}_l^w$ be the map $c$ which maps $\bpsi$ to $1$ and all other $\bpsi'$ to $0$.

    For a representation $V$ in $\catO^w$, its $q$-character is defined by
    \[\chi_q(V) = \sum_{\bpsi \in \mathfrak{r}} \dim(V_{\bpsi})[\bpsi] \in \mathcal{E}_l^w.\]
\end{definition}

We use the identification of variables $e^{\alpha}$ with $[\bpsi]$ for $\bpsi = (q^{\langle \alpha,\alpha_i \rangle})_{i \in I}$, and variables $Y_{i,a}$ with $[\bpsi]$ for $\bpsi = (1,\cdots, q_i\frac{1-q_i^{-1}au}{1-q_iau} ,\cdots,1)$. 

The structure of $q$-characters for $\catO^w$ remains to be studied. We conjecture that the $q$-character of inductive limit coincide with the projected limit conjecturally constructed in Conjecture~\ref{conj:limit of projected w norm}, and we also conjecture that their $q$-characters can always be factorized into a product of a constant part with a non-constant part. 
    \begin{conjecture}[Continue of Conjecture~\ref{conj:limit of projected w norm}]\label{conj:inductive limit and projected limit}
    	For $w \in W$, we conjecture that
    	\begin{enumerate}
    		\item for any fixed index $i \in I$ and $m$ a monomial in $Y_{j,b}$, $j \in I, b \in \Ci^*$, we have
    		\[T_w(\bpsi_{i,a}^{-1}m) \pi_{q,\infty}^w(m) = \chi_q(V_{\infty}^w(m)),\]
    		
    		\item the projected limit $\pi_{q,\infty}^w(m) \in \overline{\mathscr{CA}}_w$ factorizes into a product
    			\[\pi_{q,\infty}^w(m) = c \times a,\]
    			where $c \in \overline{\mathscr{C}}_w$ and $a \in T_w(\bpsi_{i,a}^{-1}m) \overline{\mathscr{A}}_w$.
    		
    		\item for any irreducible $w$-highest weight module $L_w(\bpsi) \in \catO^w$ of $w$-highest $l$-weight $\bpsi$, its $q$-character is an element in $\overline{\mathscr{CA}}_w$ which factorizes into a product
    		\[\chi_q(L_w(\bpsi)) = c \times a,\]
    		where 
    		\[c \in \overline{\mathscr{C}}_w\]
    		is called the constant part of $\chi_q(L_w(\bpsi))$, and 
    		\[a \in \bpsi \overline{\mathscr{A}}_w\]
    		is called the non-constant part of $\chi_q(L_w(\bpsi))$.
    	\end{enumerate}
	\end{conjecture}
    
    \begin{remark}
    	Even for representations $L(\bpsi)$ in the category $\catO$ without twist, the third statement in Conjecture~\ref{conj:inductive limit and projected limit} that $\chi_q(L(\bpsi))$ has not been proved for all $\bpsi$. Nevertheless, it is proved for a large family of representations. See Remark~\ref{remark:examples of conj decomp} for further discussion.
    \end{remark}
%--------------------------------------------------------
%SECTION
%--------------------------------------------------------
\section{Further direction}\label{sec:further direction}

%--------------------------------subsection-----------------------------%
\subsection{Relation between \texorpdfstring{$\catO^w$}{} and \texorpdfstring{$\catO$}{}}\label{subsec:relation Ow and O}
A $\Uqb$-module may lie in two categories $\catO^w$ and $\catO^{w'}$ at the same time. For example, when $\g = \mathfrak{sl}_3$, the $l$-highest weight module $L(\bpsi_{1,a}^{-1})$ lies in the category $\catO$ be definition. At the same time, $L(\bpsi_{1,a}^{-1})$ lies also in the category $\catO^{s_2}$, since its usual character is known \cite[Section~4.1]{hernandez2012asymptotic}. Therefore $L(\bpsi_{1,a}^{-1}) \simeq L_{s_2}(\bpsi_{1,a}^{-1})$. For other $w \in W$, the irreducible representations $L(\bpsi_{1,a}^{-1}) \in \catO$ and $L_w(\bpsi_{1,a}^{-1}) \in \catO^w$ parameterized by the same $l$-weight $\bpsi_{1,a}^{-1}$ are two different representations. 

A natural question is to ask the relationship between $L(\bpsi) \in \catO$ and $L_w(\bpsi) \in \catO^w$ parameterized by the same $l$-weight $\bpsi$. An interesting phenomenon is that their $q$-characters are closely related. We propose a conjecture and verify it in some examples. This relationship provides us a new approach to calculate $q$-characters of $L(\bpsi) \in \catO$ which are previously unknown. 

	\begin{definition}
		For $c \in \overline{\mathscr{C}}_{w}$, let $c^{-1} \in \overline{\mathscr{C}}_{ww_0}$ be the image of $c$ under the morphism 
		\[\overline{\mathscr{C}}_{w} \to \overline{\mathscr{C}}_{ww_0}, \quad e^{\alpha} \mapsto e^{-\alpha}, \; \forall \alpha \in \Lambda.\]
		Here $w_0$ is the longest element in $W$.
	\end{definition}
	
We restrict to the case
	\[\bpsi = T_w(\bpsi_{i,a}^{-1}).\]
Let $c$ and $a$ be the constant part and the non-constant part of the protected limit $\bpsi \pi_{q,\infty}^w$ in Conjecture~\ref{conj:inductive limit and projected limit} (Recall that this projected limit is conjectured to be $\chi_q(L_w(\bpsi))$).
    \begin{conjecture}[continue of Conjecture~\ref{conj:inductive limit and projected limit}]\label{conj:OOw differs by const}
        %$\upsilon$ maps the projected limit $\pi_{q,\infty}^w$ to $\tilde{\chi}_q(L(M_w))$.
        The product $c^{-1} \times a$ is an element in $\overline{\mathscr{CA}}_e$, and 
        \[c^{-1} \times a = \chi_q(L(T_w(\bpsi_{i,a}^{-1})))\]
    \end{conjecture}

As an application, Conjecture~\ref{conj:inductive limit and projected limit} and Conjecture~\ref{conj:OOw differs by const} provide us an algorithm to calculate the normalization of the $\widetilde{Q}$-operator in the $Q\widetilde{Q}$-systems \cite{frenkel2018spectra,frenkel2021folded,wang2023qq,frenkel2023extended}.

    \begin{example}
        Let 
        \[\widetilde{\bpsi}_{i,a} = \bpsi_{i,a}^{-1} \prod_{j: C_{i,j}=-1} \bpsi_{j,aq_i} \prod_{j: C_{i,j}=-2}\bpsi_{j,a}\bpsi_{j,aq^2} \prod_{j: C_{i,j}=-3}\bpsi_{j,aq^{-1}}\bpsi_{j,aq}\bpsi_{j,aq^3} = T_{s_iw_0}(\bpsi_{\bar{i},aq_{\bar{i}}^{2-r^{\vee}h^{\vee}}}^{-1}).\] 
        Here $h^{\vee}$ is the Coxeter number of $\g$, $r^{\vee}$ is the lacing number of the Dynkin diagram of $\g$, $w_0$ is the longest element in $W$, and $\bar{i} \in I$ is such that $w_0(\alpha_i) = -\alpha_{\bar{i}}$.
        Let $X_{i,a}$ be the $l$-highest weight representation $L(\widetilde{\bpsi}_{i,a})$ in the category $\catO$. This representation is used to construct $Q\widetilde{Q}$-systems in \cite{frenkel2018spectra,frenkel2021folded,wang2023qq}, and the constant part of $\chi_q(X_{i,a})$ is of great interest.
        
        Following Conjecture~\ref{conj:inductive limit and projected limit}, we have an algorithm to calculate the $q$-character of $L_{s_iw_0}(\widetilde{\bpsi}_{i,a})$ in the category $\catO^w$. Then as an application of Conjecture~\ref{conj:OOw differs by const}, $\chi_q(X_{i,a})$ is obtained from $\chi_q(L_{s_iw_0}(\widetilde{\bpsi}_{i,a}))$ simply by changing its constant part $c$ to $c^{-1}$.
    \end{example}

Now we verify Conjecture~\ref{conj:OOw differs by const} when $\g = \sltwo$ and $\g =\mathfrak{sl}_3$ for the projected limits $\pi_{q,\infty}^w$. We will also verify it when $w = w_0$ the longest element for general $\g$ at the end. 

	\begin{example}
		In this case, the projected limits $\pi_{q,\infty}^w$ are already calculated in Example~\ref{ex:projectedlimitsl2} and Example~\ref{ex:projectedlimitsl3}. We need to compare them with the $q$-characters of $\Uqb$-modules which are known \cite{hernandez2012asymptotic,frenkel2021folded}:
		\begin{itemize}
			\item $\g = \sltwo$, $w =s_1$:
			\[\widetilde{\chi}_q(L(\bpsi_{1,q^2})) = 1 + e^{-\alpha_1} + e^{-2\alpha_1} + \cdots, \]
			which differs from the limit $\pi_{q,\infty}^{s_1}$ in Example~\ref{ex:projectedlimitsl2} by changing $e^{\alpha_1}$ to $e^{-\alpha_1}$.
			\item $\g =\mathfrak{sl}_3$, $w=s_1$:
			\[\widetilde{\chi}_q(L(\bpsi_{1,q^2} \bpsi_{2,q}^{-1})) = (1 + e^{-\alpha_1} + e^{-2\alpha_1} + \cdots )(1 + A_{2,q}^{-1} + A_{2,q}^{-1}A_{2,q^{-2}}^{-1} + \cdots),\]
			which differs from the limit $\pi_{q,\infty}^{s_1}$ in Example~\ref{ex:projectedlimitsl3} by changing $e^{\alpha_1}$ to $e^{-\alpha_1}$.
			\item $\g =\mathfrak{sl}_3$, $w=s_2s_1$:
			\[\widetilde{\chi}_q(L(\bpsi_{2,q^3})) = \sum_{M = 0}^{\infty}\sum_{N = 0}^{M} e^{-N\alpha_1 - M \alpha_2},\]
			which differs from the limit $\pi_{q,\infty}^{s_2s_1}$ in Example~\ref{ex:projectedlimitsl3} by changing $e^{\alpha_2}$ to $e^{-\alpha_2}$ and changing $e^{\alpha_1}$ to $e^{-\alpha_1}$.
		\end{itemize}
	\end{example}
	
Conjecture~\ref{conj:OOw differs by const} does not hold for general $\bpsi$ of the form $T_w(\bpsi_{i,a}^{-1}m)$ when $m$ is a general monomial in $Y_{j,b}$. Nevertheless, it remains valid for some representations other than those mentioned in Conjecture~\ref{conj:OOw differs by const}. We verify it in the following example:
    \begin{example}\label{ex:more than KR case}
        We calculate explicitly the limits when $\mathfrak{g} = \mathfrak{sl}_3$, with the highest weight 
        \[M_km =  Y_{1,aq^{-2k+1}}\cdots Y_{1,aq^{-3}}Y_{1,aq^{-1}}Y_{2,aq^{2}}\cdots Y_{2,aq^{2l}} \]
        for a fixed $l \in \N^*$ as we studied in Example~\ref{ex:minimal affinization sl3}.
            \begin{itemize}
                \item For $w=s_1$, 
                    \[\pi_{q,\infty}^{s_1}(m) = (1 + e^{\alpha_1} + e^{2\alpha_1} + \cdots )(\sum_{N = -1}^{\infty}\sum_{M = -1}^{\min(N,l)} \prod_{j = 0}^{N}\prod_{s = 0}^{M} A_{2,q^{2l+1-2j}}^{-1} A_{1,q^{2l+2-2s}}^{-1} ).\]
                Here $\prod_{0}^{-1}$ is understood to be $1$.

                Compare this limit with the $q$-character of $\Uqb$-module $L(\bpsi_{1,q^2} \bpsi_{2,q^{2l+1}}^{-1})$:
                    \[\widetilde{\chi}_q(L(\bpsi_{1,q^2} \bpsi_{2,q^{2l+1}}^{-1})) =  (1 + e^{-\alpha_1} + e^{-2\alpha_1} + \cdots )(\sum_{N = -1}^{\infty}\sum_{M = -1}^{\min(N,l)} \prod_{j = 0}^{N}\prod_{s = 0}^{M} A_{2,q^{2l+1-2j}}^{-1} A_{1,q^{2l+2-2s}}^{-1} ).\]
                They differ by changing $e^{\alpha_1}$ to $e^{-\alpha_1}$ in the constant part.

                \item For $w=s_2s_1$, 
                    \[\pi_{q,\infty}^{s_2s_1}(m) = (\sum_{M = 0}^{\infty}\sum_{N = 0}^{M} e^{N\alpha_1 + M \alpha_2})(1 + A_{1,q^{2l+2}}^{-1} + \cdots + \prod_{s=0}^{l}A_{1,q^{2l+2-2s}}^{-1} ). \]
                Compare this limit with the $q$-character of $\Uqb$-module $L(\bpsi_{1,q^2} \bpsi_{1,q^{2l+2}}^{-1} \bpsi_{2,q^{2l+3}})$: 
                    \[\widetilde{\chi}_q(L(\bpsi_{1,q^2} \bpsi_{1,q^{2l+2}}^{-1} \bpsi_{2,q^{2l+3}})) = (\sum_{M = 0}^{\infty}\sum_{N = 0}^{M} e^{-N\alpha_1 - M \alpha_2})(1 + A_{1,q^{2l+2}}^{-1} + \cdots + \prod_{s=0}^{l}A_{1,q^{2l+2-2s}}^{-1} ).\]
                They differs by changing $e^{\alpha_1}$ to $e^{-\alpha_1}$ and changing $e^{\alpha_2}$ to $e^{-\alpha_2}$ in the constant part.
    
                %\item For $T_1T_2T_1$ acting on $V_{k,l}$,\[\lim \frac{\chi_q(T_1^*T_2^*T_1^*V_{k,l})}{T_1T_2T_1M_{k,l}} =  \]
            \end{itemize}
    \end{example}

	\begin{remark}
		The above conjecture states a possible relation between $\chi_q(L_w(\bpsi))$ and $\chi_q(L(\bpsi))$, which are classified by the same $l$-weight $\bpsi$. Another further direction is to study the relation between $\chi_q(L_w(T_w(\bpsi)))$ and $\chi_q(L(\bpsi))$, which are classified by two $l$-weights differing by the braid group action on $\mathfrak{r}$. It is expected that the later relation is connected to the Weyl group action of Frenkel-Hernandez, as discussed in \cite[Section~4.3]{frenkel2023extended}. 
	\end{remark}

%------------------------------subsection-------------------------------%

\subsection{Relation with shifted quantum affine algebras}\label{sec:shiftedQAA}
Shifted quantum affine algebras are defined to study Coulomb branches of 3d $N= 4$ SUSY gauge theories \cite{braverman2019coulomb,finkelberg2019multiplicative}. The representation theory of shifted quantum affine algebras have very rich structure and are found related to many different domains \cite{hernandez2023representations,frenkel2023extended,geiss2024representations}. For example, the representations of shifted quantum affine algebras are related to the representations of quantum affine Borel algebras $\Uqb$. One approach to this relation is the induced representation defined in \cite[Section~7.2]{hernandez2023representations}.

Here we propose a possible explanation of the third point in Conjecture~\ref{conj:inductive limit and projected limit} for representations in $\catO$, via shifted quantum affine algebras. 

More precisely, we conjecture that the non-constant part $a$ coincides with $q$-character of representations of shifted quantum affine algebras, and the constant part $c$ can be given by an explicit formula below.

For any coweight $\mu \in P^{\vee} = \oplus_{i \in I} \Z \omega_i^{\vee}$, let $U_q^{\mu}$ be the shifted quantum affine algebras defined as in \cite[Section~3.1]{hernandez2023representations}, where the notion of $l$-highest weight representations of $U_q^{\mu}$ are also defined. These representations are classified by the highest $l$-weight $\bpsi$, such that $\psi_i(u)$ is a rational function of degree $\alpha_i(\mu)$ \cite[Theorem~4.13]{hernandez2023representations}. Denote by $L^{\mu}(\bpsi)$ the irreducible $l$-highest weight representation of $U_q^{\mu}$.

\begin{conjecture}\label{conj:nonconst part sQAA}
	Let $L(\bpsi)$ be an irreducible $l$-highest weight module of $\Uqb$. Let $L^{\mu}(\bpsi)$ be the irreducible $l$-highest weight representation of the shifted quantum affine algebra $U_q^{\mu}$, where $\mu \in P^{\vee}$ is the determined by the degree of $\bpsi$. Then
	\[\chi_q(L(\bpsi)) = c \times a,\]
	where
	\begin{equation}\label{eq:explicit formula c}
	c = \prod_{\alpha \in \Delta_+}(\frac{1}{1-e^{-\alpha}})^{\max(0,\alpha(\mu))}, \end{equation}
	and 
	\[a = \chi_q(L^{\mu}(\mathbf{\Psi})).\]
\end{conjecture}

\begin{remark}\label{remark:examples of conj decomp}
	We remark that this conjecture is proved true when $\mu \in P^{\vee}_- = \oplus_{i \in I} \Z_{\leq 0} \omega_i^{\vee}$, since in this case the shifted quantum affine algebra contains the Borel subalgebras as a subalgebra, and the restriction preserves irreducibility \cite[Corollary~4.12]{hernandez2023representations}. In this case, the constant part $c=1$. 
	
	It is also proved true when $\bpsi$ is a monomial in $\bpsi_{i,a}$,  $i \in I, a \in \Ci^*$. In this case, the non-constant part $a = \bpsi$, and the constant part $c$ is explicit as in \eqref{eq:explicit formula c} \cite{hatayama1999remarks,li2016graded,lee2021product}. More generally, it still holds when $\bpsi$ is a monomial in $\bpsi_{i,a}$ and $Y_{i,a}$, $i \in I, a \in \Ci^*$ \cite[Section~5.2]{feigin2017finite}.
\end{remark}

\begin{proposition}\label{prop:longest w0}
	Conjecture~\ref{conj:OOw differs by const} holds when $w = w_0$ the longest element in $W$.
\end{proposition}
\begin{proof}
	When $w = w_0$ the longest element in Weyl group, it is known that $T_{w_0} M_k$ equals the $l$-lowest weight of KR-modules, which is given by $ Y_{\bar{i},aq_{\bar{i}}^{h^{\vee}r^{\vee}-1}}^{-1} Y_{\bar{i},aq_{\bar{i}}^{h^{\vee}r^{\vee}-3}}^{-1} \cdots Y_{\bar{i},aq_{\bar{i}}^{h^{\vee}r^{\vee}-2k+1}}^{-1}$ \cite{chari1996minimalsimp}.  Here recall that $\bar{i} \in I$ is such that $w_0(\alpha_i) = -\alpha_{\bar{i}}$. 
	
	By Frenkel-Mukhin algorithm \cite{frenkel2001combinatorics}, for any fixed $\lambda \in \Lambda$ and $R \in \Z$, for $k$ large enough, all $w_0$-normalized monomials in any weight space $V_{w_0(k\omega_i) + \lambda}$ are projected to $e^{\lambda}$ under the morphism $\pi_R$ in \eqref{eq:piR}. In particular, for any fixed $R$, the projected $w_0$-normalized $q$-characters converge to $\pi_{q,\infty}^{w_0,\geq R}$ when $k \to \infty$, and this limit is an element in $\Z \llbracket e^{\alpha_i} \rrbracket = \overline{\mathscr{C}}_{w_0}$, which is in fact independent of $R$.
	
	As a consequence, the projected limit  $\pi_{q,\infty}^{w_0}$ simply equals to $\lim_{k \to \infty} e^{-w_0(k\omega_i)}\chi(L(M_k))$.
	
	Since the usual characters of finite-dimensional representations $\chi(L(M_k))$ are invariant under Weyl group actions, we have 
	\[\pi_{q,\infty}^{w_0} = \lim_{k \to \infty} w_0(e^{-k\omega_i}\chi(L(M_k))) = w_0(\chi(L_{i,a}^-)) = w_0(\prod_{\alpha \in \Delta_+}(\frac{1}{1-e^{-\alpha}})^{\alpha(\omega_i^{\vee})}) = \prod_{\alpha \in \Delta_+}(\frac{1}{1-e^{\alpha}})^{\alpha(\omega^{\vee}_{\bar{i}})}.\]
	
	In particular, it differs from
	\[ \widetilde{\chi}_q (L_{\bar{i},aq^{h^{\vee}r^{\vee}}}^+) = \prod_{\alpha \in \Delta_+}(\frac{1}{1-e^{-\alpha}})^{\alpha(\omega^{\vee}_{\bar{i}})}\]
	by changing all $e^{-\alpha_i}$ to $e^{\alpha_i}$.
\end{proof}

\bibliographystyle{alpha-abbrev}
\bibliography{bib}

\end{document}